\documentclass[a4paper,reqno]{amsart}
\usepackage{amssymb}
\usepackage{array}
\usepackage{calc}
\usepackage{pstricks}

\begin{document}
\title[A refinement of B\'ezout's Lemma]{A refinement of B\'ezout's Lemma, and order~3 elements in some quaternion algebras over~$\Q$}

\author[Donald I.~Cartwright, Xavier Roulleau]{Donald I.~Cartwright and Xavier Roulleau}

\addtolength{\baselineskip}{-0.5pt}
\numberwithin{equation}{section}
\newtheorem{theorem}{Theorem}[section]
\newtheorem{proposition}{Proposition}[section]
\newtheorem{lemma}{Lemma}[section]
\newtheorem{corollary}{Corollary}[section]
\theoremstyle{definition}
\newtheorem{definition}{Definition}
\theoremstyle{remark}
\newtheorem{remark}{Remark}[section]
\newtheorem{example}{Example}
\newcommand{\F}{{\mathbb F}}
\newcommand{\Q}{{\mathbb Q}}
\newcommand{\Z}{{\mathbb Z}}
\newcommand{\cG}{\mathcal{G}}
\newcommand{\cH}{\mathcal{H}}
\newcommand{\cL}{\mathcal{L}}
\newcommand{\cO}{\mathcal{O}}
\newcommand{\Trace}{\mathrm{Trace}}
\newcommand{\myj}{j}
\newcommand{\myNrd}{{\mathrm{Nr}}}
\newcommand{\myTr}{{\mathrm{Tr}}}
\begin{abstract}Given coprime positive integers $d',d''$, B\'ezout's Lemma
  tells us that there are integers $u,v$
  so that $d'u-d''v=1$. We show that, interchanging $d'$ and~$d''$ if necessary,
  we may choose $u$ and~$v$ to be {\it Loeschian numbers\/}, i.e., of the
  form $|\alpha|^2$, where $\alpha\in\Z[\myj]$, the ring of integers of the number field~$\Q(\myj)$,
  where $\myj^2+\myj+1=0$.  We do this by using {\it Atkin-Lehner
  elements\/} in some quaternion algebras~$\cH$. We use this fact to count the
  number of conjugacy classes of elements of order~3 in an order $\cO\subset\cH$.
\end{abstract}
\maketitle

\begin{section}{Introduction}\label{sec:preliminaries}The result in elementary number theory commonly called B\'ezout's Lemma seems
  to have been first formulated by Bachet in~1624. For coprime $d',d''\ge1$, if $d'u-d''v=1$, then replacing $(u,v)$ by $(u+d''m,v+d'm)$
  for large enough $m$, we may assume that $u,v\ge0$. Consider the field $\Q(\myj)$, where $\myj^2+\myj+1=0$,
  with non-trivial automorphism $x\mapsto\bar x$.
  The ring of integers of~$\Q(\myj)$ is $\mathfrak o=\Z[\myj]$. For reasons we explain below, we would like to choose the
  $u$ and~$v$ of B\'ezout's Lemma to be in the set $\cL$ of Loeschian integers, the integers of the form $x\bar x=a^2-ab+b^2$,
  where $x=a+b\myj\in\mathfrak o$.
  It is not always possible to choose $u,v\in\cL$. For example, if
  $(d',d'')=(5,3)$, then $d'u-d''v=1$ if and only if $u=2+3m$ and
  $v=3+5m$. As $a^2-ab+b^2=(a+b)^2-3ab$, any $u\in\cL$ is congruent to~0 or~1 (mod~3).
    So $2+3m\not\in\cL$
  for any integer $m\ge0$. For $(d',d'')=(3,5)$, however, $d'u-d''v=1$ holds for $u=7\in\cL$ and $v=4\in\cL$.
  In fact, $u,v$ can be always be chosen in~$\cL$ if we are allowed to interchange $d'$ and~$d''$.
  Equivalently, (a)~of the following holds.
  \begin{theorem}\label{thm:bezoutresult}Suppose that $d',d''$ are coprime positive integers. Then
    \begin{itemize}
    \item[(a)] there are $u,v\in\cL$ so that $d'u-d''v\in\{-1,1\}$,
    \item[(b)] if $d'd''\equiv2$ (mod~3), then the $u,v$ in~(a) can be chosen so that $3\nmid uv$,
    \item[(c)] if $3\nmid d'd''$, then the $u,v$ in~(a) can be chosen so that $d'u-d''v=1$.
      \end{itemize}
  \end{theorem}

   As $u,v\equiv0$ or~1 (mod~3) for $u,v\in\cL$, if $d'd''\equiv1$ (mod~3), it is not possible for $d'u-d''v\in\{-1,1\}$ and $3\nmid uv$ to
    hold. We have been unable to prove the following:

  \smallskip\noindent{\bf Conjecture 1:} When $d'd''\equiv0$ (mod~3), the $u,v$ in Theorem~\ref{thm:bezoutresult}~(a) can be chosen so that $3\nmid uv$.

   \smallskip A prime number $p$ is in~$\cL$ if $p\equiv1$ (mod~3), as $p$ then splits in~$\Q(\myj)$
    (see \cite[V(1.1)]{frohlichtaylor}, for example), or if $p=3$, which ramifies in~$\Q(\myj)$. If $p\equiv2$ (mod~3) then $p\not\in\cL$, and $p$ is inert in~$\Q(\myj)$.
      As $uv\in\cL$ when $u,v\in\cL$, we see that $\cL$ contains all the numbers of the form
     $3^am_1m_2^2$,
  where $a\ge0$ and $m_2$ are integers, and where $m_1$ is a product of distinct prime numbers, all congruent to~1 (mod~3); 
  it is not hard to verify all $u\in\cL$ are of this form.

  Throughout this paper, $d\ge1$ is an integer, which will later be taken to be~$d'd''$,
  and $\cH$ is the quaternion algebra consisting of elements $\xi=x+y\phi$, where
  $x,y\in\Q(\myj)$, with $\phi^2=d$ and with $\phi x=\bar x\phi$
  for all $x\in\Q(\myj)$. There is an injective algebra homomorphism $\Psi:\cH\to M_2(\Q(\myj))$ mapping
    $\xi=x+y\phi$ to
    \begin{displaymath}
      \begin{pmatrix}
        x&y\\d\bar y&\bar x
        \end{pmatrix}.
    \end{displaymath}
    We shall identify $\cH$ with its image in~$M_2(\Q(\myj))$. The discriminant of $\cH$ equals $d_\cH$ if $ d_\cH\equiv1$ (mod~3) and
     $3d_\cH$ otherwise, where $d_\cH$ is the product of the primes $p$ dividing $d$ 
     to an odd power and such that $p\equiv 2$ (mod~3).
    
    The reduced norm $\myNrd(\xi)$ and reduced trace $\myTr(\xi)$ of $\xi=x+y\phi\in\cH$ are
    by definition $\det(\Psi(\xi))=x\bar x-dy\bar y$ and $\Trace(\Psi(\xi))=x+\bar x$, respectively.

    Writing $\xi=x_1+x_2\myj+x_3\phi+x_4\myj\phi$, with $x_1,x_2,x_3,x_4\in\Q$, we have
    \begin{displaymath}
      \myNrd(\gamma)=x_1^2-x_1x_2+x_2^2-d(x_3^2-x_3x_4+x_4^2)\quad\text{and}\quad\myTr(\gamma)=2x_1-x_2.
      \end{displaymath}
    
    The inverse of $\xi=x+y\phi\in\cH$ exists if and only if $\myNrd(\xi)\ne0$, in which case
    \begin{displaymath}
      \xi^{-1}=\frac{1}{\myNrd(\xi)}\bigl(\bar x-y\phi\bigr).
    \end{displaymath}
    Let $\cO$ denote the order of~$\cH$ consisting of elements $x+y\phi$ with $x,y\in\mathfrak o$ (its discriminant is $3d$).
    Let $\cO^\times$ denote the set of elements of~$\cO$ which are invertible, with inverse in~$\cO$. Then
    $\cO^\times=\{\xi\in\cO:\myNrd(\xi)=1\ \text{or}\ -1\}$.

   We obtain our refinement of B\'ezout's Lemma by constructing {\it Atkin-Lehner\/} elements in~$\cH$.
    We call a divisor $d'$ of~$d$ a {\it Hall divisor\/} of~$d$
    if $d'$ is coprime to $d''=d/d'$, and we shall then write $d'\|d$. Suppose that $u=x\bar x\in\cL$
    and $v=y\bar y\in\cL$ and $d'u-d''v=\epsilon\in\{-1,1\}$. Then we show in Lemma~\ref{lem:AtkinLehnerconstruction}
    that $w_d'=d'x+y\phi$ is an Atkin-Lehner element: it normalizes~$\cO$, and $w_{d'}^2=d'q$ for some $q\in\cO^\times$.

    We apply the above to the study of elements of~$\cO$ having order~3, i.e., satisfying  $\xi^3=1$ and $\xi\ne1$.
    Given two such elements $\xi,\eta$, call them ($\cO^\times$-){\it conjugate\/} if there is an
    $\alpha\in\cO^\times$ such that $\eta=\alpha\xi\alpha^{-1}$.

\begin{theorem}\label{thm:Cdcount}Let $d\ge1$ be an integer, and let $d=p_1^{m_1}\cdots p_r^{m_r}$
  be the factorization of $d$ into powers of distinct primes (so when $d=1$ we have
  $r=0$). Let $C_d$ denote the number of $\cO^\times$-conjugacy classes
  of elements of order~3. Then
  \begin{itemize}
  \item[(a)] if $p_i\ne3$ for each $i$, then $C_d=2^r$;
  \item[(b)] if some $p_i=3$ and $m_i\ge2$, then $C_d$ is either $2^r$ or~$3\times2^r$;
  \item[(c)] if some $p_i=3$ and $m_i=1$, let $\tilde d=d/3$.
    \begin{itemize}
    \item[(i)] if $\tilde d\equiv1$ (mod~3), then $C_d=2^r$,
    \item[(ii)] if $\tilde d\equiv2$ (mod~3), then $C_d=2^{r+1}$.
    \end{itemize}
  \end{itemize}
\end{theorem}
 In~(b), we have been able to show that $C_d=3\times2^r$ in many cases (see Section~\ref{sec:conjectures}), but have been unable
to show this in general, so we formulate:

\medskip\noindent{\bf Conjecture 2:} If $d$ is as in Theorem~\ref{thm:Cdcount}(b), then
$C_d$ is always~$3\times2^r$.

\smallskip This conjecture holds for~$d$ if and only
if Conjecture~1 holds for coprime $d',d''\ge1$ such that $d'd''=d/3$ (see Proposition~\ref{prop:compareconjectures}).
   
\smallskip To the order $\cO$ is classically associated a Shimura curve $\mathcal{X}(\cO)$, 
which is a moduli space for abelian surfaces with quaternionic multiplication by~$\cO$.
The number~$C_d'$ of conjugation classes by $\cO^\times$ of order~3 subgroups of $\cO^\times$ enters into 
the computation of the genus of the curve $\mathcal{X}(\cO)$. From the knowledge of~$C_d$, one can easily obtain~$C_d'$
(see Remark~\ref{rem:effectofsquaring}(b)). When $\cO$ is an Eichler 
order, $C_d'$ is known (see \cite[Chapitre~V, Section 3]{Vigneras} for references). Our formulas generalise 
that result to a large class of orders,  
since in fact every order containing the Eisenstein integers is either of the form~$\cO$ (for some~$d$), 
or contains such an order~$\cO$ with index~3 (see \cite{Roulleau}). 
The number~$C_d'$ is also linked to the problem of computing the number of generalised
 Kummer structures on generalised Kummer surfaces (see \cite{Roulleau}).
  
  {\bf Acknowledgements} The second author thanks A. Dz\v{a}mbi\'c, O. Ramar\'e and M. Stover for useful email
   exchanges on quaternion algebras or B\'ezout's Lemma. 
  \end{section}
\begin{section}{Atkin-Lehner elements and the proof of Theorem~\ref{thm:bezoutresult}(a)}\label{sec:ALelts}
  Suppose $d\ge1$ is an integer, and $d'$ is a Hall divisor of~$d$, and write $d''=d/d'$.
  Suppose that $u=x\bar x\in\cL$, $v=y\bar y\in\cL$ and $d'u-d''v=\epsilon\in\{-1,1\}$. Form
      \begin{equation}\label{eq:atkinlehner}
        w_{d'}=d'x+y\phi\in\cO.
      \end{equation}
      Then
    \begin{displaymath}
    \myNrd(w_{d'})={d'}^2x\bar x-dy\bar y=d'(d' x\bar x-d''y\bar y)=d'(d'u-d''v)=d'\epsilon.
    \end{displaymath}
    Moreover,
    \begin{displaymath}
      w_{d'}^2=d'\bigl(d'x^2+d''y\bar y+y(x+\bar x)\phi\bigr)=d'q,
    \end{displaymath}
    where $q\in\cO$. Also, $(d'\epsilon)^2=\myNrd(w_{d'}^2)={d'}^2\myNrd(q)$,
    so that $\myNrd(q)=1$, and $q\in\cO^\times$.

    Unless $d'=1$, $w_{d'}$ is not in~$\cO^\times$. Nevertheless, we have
    \begin{lemma}\label{lem:AtkinLehnerconstruction}The element $w_{d'}$ normalizes~$\cO$. That is,
      conjugation $\alpha\mapsto w_{d'}\alpha w_{d'}^{-1}$ is
      an automorphism of~$\cO$. If $x',y'\in\mathfrak o$ also satisfy
      $d'x'\bar x'-d''y'\bar y'=\epsilon'\in\{-1,1\}$, and if
      $w_{d'}'=d'x'+y'\phi$, then $w_{d'}'w_{d'}^{-1}\in\cO^\times$ and $w_{d'}^{-1}w_{d'}'\in\cO^\times$,
      so that the two-sided ideal $w_{d'}\cO=\cO w_{d'}$ of~$\cO$ does not depend
      on the particular choice of $x,y\in\mathfrak o$ satisfying $d'x\bar x-d''y\bar y\in\{-1,1\}$.
      Moreover, if $\alpha\in\cO^\times$, then $\alpha w_{d'}$ can be written $d'x'+y'\phi$,
      where $x',y'\in\mathfrak o$ also satisfy $d'x'\bar x'-d''y'\bar y'\in\{-1,1\}$.
    \end{lemma}
    \begin{proof}If $\alpha=a+b\phi\in\cO$, then $w_{d'}\alpha w_{d'}^{-1}$ equals
         \begin{displaymath}
           \begin{aligned}
             &(d'x+y\phi)(a+b\phi)\frac{1}{d'\epsilon}(d'\bar x-y\phi)\\
             &=\epsilon\bigl(x(d'a\bar x-db\bar y)+d''y(d'\bar b\bar x-\bar a\bar y)\bigr)
             +\bigl(x(d'bx-ay)+y(\bar ax-d''\bar b y)\bigr)\phi\in\cO.\\
             \end{aligned}
         \end{displaymath}
         Also,
                 \begin{displaymath}
 w_{d'}^{-1}\alpha w_{d'}=\bigl(\frac{1}{d'}q^{-1}w_{d'}\bigr)\alpha w_{d'}=q^{-1}w_{d'}\alpha\bigl(\frac{1}{d'}w_{d'}q^{-1}\bigr)q=q^{-1}(w_{d'}\alpha w_{d'}^{-1})q\in\cO,
         \end{displaymath}
         so that conjugation by~$w_{d'}$ is an automorphism of~$\cO$. Moreover,
                \begin{displaymath}
           \begin{aligned}
             w_{d'}'w_{d'}^{-1}=(d'x'+y'\phi)\frac{1}{d'\epsilon}(d'\bar x-y\phi)
             &=\epsilon\bigl(d'x'\bar x-d''y'\bar y+(xy'-x'y)\phi\bigr)\in\cO.\\
             \end{aligned}
                \end{displaymath}
                Since $\myNrd(w_{d'}'w_{d'}^{-1})=d'\epsilon'(d'\epsilon)^{-1}=\epsilon'/\epsilon\in\{-1,1\}$,
                we have $w_{d'}'w_{d'}^{-1}\in\cO^\times$. Writing ${w_{d'}'}^2=d'q'$, with $q'\in\cO^\times$,
                \begin{displaymath}
             w_{d'}^{-1}w_{d'}'=\bigl(\frac{1}{d'}q^{-1}w_{d'}\bigr)w_{d'}'=q^{-1}w_{d'}\bigl(\frac{1}{d'}w_{d'}'{q'}^{-1}\bigr)q'=q^{-1}(w_{d'}{w_{d'}'}^{-1})q'
             \in\cO,
                \end{displaymath}
                so that $w_{d'}\cO$ does not depend on the particular choice of $x,y$ satisfying $d'x\bar x-d''y\bar y\in\{-1,1\}$.
                To prove the last statement, we calculate that $\alpha w_{d'}$ equals
                \begin{displaymath}
                  d'(ax+d''b\bar y)+(ay+d'b\bar x)\phi,
                \end{displaymath}
                and verify for $x'=ax+d''b\bar y$ and $y'=ay+d'b\bar x$ that $d'x'\bar x'-d''y'\bar y'=\epsilon\myNrd(\alpha)$.              
    \end{proof}
    \begin{definition}
        We call $w\in\cO$ an {\it Atkin-Lehner element\/} of~$\cO$ if $w^2=mq$ for some $q\in\cO^\times$ and nonzero $m\in\Z$,
        and if also $w\cO w^{-1}=\cO$.   An  Atkin-Lehner element $w$ is said trivial 
    if $w\in n\cO^\times$ for some nonzero $n\in\Z$. 
    \end{definition}
    
    The elements $w_{d'}$ with $d'\neq 1$ are non-trivial Atkin-Lehner elements of~$\cO$.

    When a prime $p$ does not split, $\Q_p(\myj):=\Q(\myj)\otimes_\Q \Q_p$ is a degree~2 field extension of~$\Q_p$ whose
    ring of integers is~$\Z_p[\myj]$.
    \begin{lemma}\label{lem:pnonsplitlocalBezout}Let $d',d''\ge1$ be coprime integers. If  a prime~$p$ does not split, there exist
      $x,y\in\Z_p[\myj]$ such that $d'x\bar x-d''y\bar y\in\{-1,1\}$. 
    \end{lemma}
    \begin{proof}First suppose that $p\ne3$. For $x_1,x_2,y_1,y_2\in\Q_p$, define
      \begin{displaymath}
        Q_p(x_1,x_2,y_1,y_2)=\begin{cases}d'(x_1^2-x_1x_2+x_2^2)-d''(y_1^2-y_1y_2+y_2^2)&\text{if}\ p\nmid d'd'',\\
        -d''(y_1^2-y_1y_2+y_2^2)&\text{if}\ p\mid d',\\
        d'(x_1^2-x_1x_2+x_2^2)&\text{if}\ p\mid d''.
        \end{cases}
      \end{displaymath}
      The quadratic form $Q_p$ has rank 4, 2 and~2, respectively, in these three cases. By \cite[Chapitre~IV, Section~1.7]{serrebook}
      the form $Q_p$ modulo~$p$ represents any element of the finite field~$\F_p$, so in particular,
      $Q_p$ represents 1 and~$-1$. The discriminant of~$Q_p$ is $-9d'd''$,
      $-3d''$ and~$3d'$, respectively, which is invertible in~$\Z_p$ in each of the three cases.
      Therefore by \cite[Chapitre~II, Corollaires~2 \&~3]{serrebook}, $Q_p$ represents 1 and~$-1$ in~$\Z_p$.
      If $Q_p(x_1,x_2,y_1,y_2)\in\{-1,1\}$, let $x=x_1+x_2\myj$ and $y=y_1+y_2\myj$.

      Now suppose that $p=3$. If 3 does not divide~$d'd''$, the same arguments with the rank~2 form
      $Q_p(x_1,y_1)=d'x_1^2-d''y_1^2$ shows that $Q_p$ represents 1 and~$-1$ in~$\Z_3$.

      Suppose now that 3 divides $d''$, so that $\gcd(3,d')=1$. Write $d''=3\tilde d$ for some integer~$\tilde d$,
      and consider the equation
      \begin{displaymath}
        d'u-3\tilde d v=\epsilon,
      \end{displaymath}
      where $\epsilon\in\{-1,1\}$ and where $u,v\in\cL$. Then $\gcd(3,u)=1$ must hold, and so $u\in\cL$ implies
      that $u\equiv1$ (mod~3). So $d'\equiv\epsilon$ (mod~3) must hold. Since $d'$ is invertible in~$\Z_3$,
      and $\epsilon/d'\equiv1$ (mod~3), there must exist an $\alpha\in\Z_3$ such that $\alpha^2=\epsilon/d'$
      (see, e.g., \cite[Chapitre~II, Section~3.3, Corollaire]{serrebook}), and then $d'x\bar x-d''y\bar y=\epsilon$
      for $(x,y)=(\alpha,0)$.

      Suppose instead that 3 divides~$d'$, so that $\gcd(3,d'')=1$. Since solving the equation $d'x\bar x-d''y\bar y=\pm1$
      is equivalent to solving the equation $d''x\bar x-d'y\bar y=\mp1$, we are reduced to the previous case.
    \end{proof}

    For each prime~$p$, let $\cH_p=\cH\otimes_\Q\Q_p$ and $\cO_p=\cO\otimes_\Z\Z_p$.
    We define a two sided ideal~$I_p=w_p\cO_p$ of~$\cO_p$ as follows: 

    \smallskip
    \noindent1. (Split case). When $p$ splits,
    there is a $\myj_p\in\Z_p$ satisfying $\myj_p^2+\myj_p+1=0$. As $p\ne3$, the map
    $(a_1,a_2)\mapsto(a_1+a_2\myj_p,a_1+a_2\myj_p^2)$ is a $\Q_p$-linear bijection
    of~$\Q_p^2$ inducing a bijection $\Z_p^2\to\Z_p^2$. 
    So if we follow the above embedding $\Psi:\cH\to M_2(\Q(\myj))$ by the embedding $M_2(\Q(\myj))\to M_2(\Q_p)$
    mapping $\myj$ to~$\myj_p$ entrywise, we get an isomorphism $\Psi_p:\cH_p\cong M_2(\Q_p)$.
    This in turn induces an embedding $\cO_p\to M_2(\Z_p)$, whose image is the subring
    $\{(a_{ij})\in M_2(\Z_p):a_{21}\in d\Z_p\}$ of~$M_2(\Z_p)$.

    Choose any $x,y,z,t\in\Z_p$ for which $d'xt-d''yz\in\{-1,1\}$. Such elements
    exist by B\'ezout's Lemma. Let $w_p\in\cO_p$ be the preimage of the matrix
    \begin{displaymath}
      \begin{pmatrix}
        d'x&y\\
        dz&d't\\
        \end{pmatrix}
      \end{displaymath}
    under the above embedding. Then $w_p\cO_p=\cO_pw_p$, $w_p^2=d'q_p$ for some $q_p\in\cO_p^\times$, and $I_p=w_p\cO_p$
    is a two-sided ideal of~$\cO_p$ which doesn't depend on the particular choice we made for $x,y,z,t$. The proof of
    this is very similar to that of Lemma~\ref{lem:pnonsplitlocalBezout}.

    \smallskip\noindent2. (Non-split case). When $p$ does not split,
    $\cH_p$ is the quaternion algebra $\{a+b\phi:a,b\in\Q_p(\myj)\}$ over~$\Q_p(\myj)$. The embedding
    $\Psi:\cH\to M_2(\Q(\myj))$ induces an embedding $\cH_p\to M_2(\Q_p(\myj))$, and the image of
    $\cO_p$ under this embedding is contained in $\{(a_{ij})\in M_2(\Z_p[\myj]):a_{21}\in d\Z_p[\myj]\}$.

    By Lemma~\ref{lem:pnonsplitlocalBezout}, we can choose $x,y\in\Z_p[\myj]$ so that $d'x\bar x-d''y\bar y\in\{-1,1\}$,
    and let $w_p=d'x+y\phi\in\cO_p$. Again, $w_p\cO_p=\cO_pw_p$, $w_p^2=d'q_p$ for some $q_p\in\cO_p^\times$, and $I_p=w_p\cO_p$
    is a two-sided ideal of~$\cO_p$ which doesn't depend on the particular choice we made for $x$ and~$y$.

    In both the split and non-split case, $I_p=\cO_p$ for all primes $p$ which do not divide~$d$, as
    then $\myNrd(w_p)=\pm d'$ is invertible in~$\Z_p$.

    For an $\cO$-lattice $\Lambda\subset\cH$, we denote by~$\Lambda_p$ the $\cO_p$-lattice
    generated by~$\Lambda$ in~$\cH_p$. We shall use the following two results:

    \begin{proposition}\label{prop: alsinabayer}(\cite[Proposition~1.57, p.~13]{alsinabayer}) Let $\Lambda'\subset\cH$ be an $\cO$-lattice.
      For each prime~$p$, let $L_p$ be a $\cO_p$-lattice in~$\cH_p$ such that $L_p=\Lambda_p'$ for all but a finite
      number of~$p$. Then there exists an $\cO$-lattice $\Lambda$ in~$\cH$ such that $\Lambda_p=L_p$ for all~$p$.
    \end{proposition}
    \begin{theorem}\label{thm:idealsareprincipal}(\cite[Theorem 3]{Roulleau}) Any ideal $\Lambda$ in the order~$\cO$ is principal.
    \end{theorem}
    \begin{proof}[Proof of Theorem~\ref{thm:bezoutresult}(a)]Let $d=d'd''$, for the $d'$ and~$d''$ of the theorem.
      Let $\cH$ and~$\cO\subset\cH$ be as in Section~\ref{sec:preliminaries}. Let
      $\Lambda'=\cO$, and for each prime~$p$, let $L_p=I_p=w_p\cO_p$, as defined above. Since $L_p=\cO_p$ for every
      prime~$p$ not dividing~$d$, the hypotheses of Proposition~\ref{prop: alsinabayer} are satisfied,
      so there is an $\cO$-lattice $\Lambda\subset\cH$ such that $I_p=\Lambda_p$ for all~$p$. Then
      $\Lambda\subset\cO$ since $\Lambda_p\subset\cO_p$ for all~$p$. By Theorem~\ref{thm:idealsareprincipal},
      there is a $\xi\in\cO$ so that $\Lambda=\xi\cO$, and let $\xi_p$ denote its image in~$\cO_p\subset\cH_p$.
      When $p$ divides~$d$, $\xi_p\cO_p=\Lambda_p=w_p\cO_p$, so that $w_p^{-1}\xi_p\in\cO_p^\times$. Hence $\xi_pw_p^{-1}=w_p(w_p^{-1}\xi_p)w_p^{-1}\in\cO_p^\times$,
      since conjugation by~$w_p$ is an automorphism of~$\cO_p$. Therefore $\cO_p\xi_p=\cO_pw_p=w_p\cO_p=\xi_p\cO_p$. If $p$ does not divide~$d$,
      then $\xi_p\cO_p=\Lambda_p=I_p=\cO_p$, so that $\xi_p\in\cO_p^\times$, so that again $\cO_p\xi_p=\xi_p\cO_p$.
      Hence $\xi\cO=\cO\xi$ because this is true locally. When $p$ divides~$d$, then
      $\xi_p^2\cO_p=\xi_p(w_p\cO_p)=\xi_p(\cO_pw_p)=w_p\cO_pw_p=w_p^2\cO_p=d'\cO_p$. When $p$ does not divide~$d$, then $\xi_p\in\cO_p^\times$,
      so that $\xi_p^2\cO_p=\cO_p=d'\cO_p$, as $d'$ is an invertible element of~$\Z_p$. So $\xi^2\cO=d'\cO$ because this is true locally.
      So $\xi^2=d'q$ for some $q\in\cO^\times$. Thus $\myNrd(\xi)^2={d'}^2\myNrd(q)$, so that $\myNrd(q)=1$ and $\myNrd(\xi)=\epsilon d'$
      for some $\epsilon\in\{-1,1\}$. Write $\xi=x_1'+x_2'\myj+y_1\phi+y_2\myj\phi$, where $x_1',x_2',y_1,y_2\in\Z$.
      We show that $d'$ divides~$x_1'$ and~$x_2'$ as follows. If $p$ divides~$d'$, suppose that $p^m\|d'$.
      Write $\xi_p=w_pq_p$, where $q_p\in\cO_p^\times$.
       
      In the split case, in terms of matrices, we have
      \begin{displaymath}
        \begin{pmatrix}x_1'+x_2'\myj_p&y_1+y_2\myj_p\\
          d(y_1+y_2\myj_p^2)&x_1'+x_2'\myj_p^2\\
        \end{pmatrix}
        =
        \begin{pmatrix}d'x_p&y_p\\
          dz_p&d't_p\\
        \end{pmatrix}
        \begin{pmatrix}q_{11}&q_{12}\\
          dq_{21}'&q_{22}\\
        \end{pmatrix},
      \end{displaymath}
      where $x_p,y_p,z_p,t_p\in\Z_p$ satisfy $d'x_pt_p-d''y_pz_p\in\{-1,1\}$, and where the entries $q_{11}$, $q_{12}$, $q_{21}'$
      and $q_{22}$ are also in~$\Z_p$, and
      satisfy $q_{11}q_{22}-dq_{21}'q_{21}\in\Z_p^\times$. Comparing $(1,1)$ and $(2,2)$ entries, we see that $x_1'+x_2'\myj_p,x_1'+x_2'\myj_p^2\in d'\Z_p$,
      and so (using $p\ne3$) we see that $x_1',x_2'\in d'\Z_p\cap\Z=p^m\Z_p\cap\Z$, so that $p^m$ divides~$x_1'$ and~$x_2'$. 

      In the non-split case,
      \begin{displaymath}
        \begin{pmatrix}x_1'+x_2'\myj&y_1+y_2\myj\\
          d(y_1+y_2\myj^2)&x_1'+x_2'\myj^2\\
        \end{pmatrix}
        =
        \begin{pmatrix}d'x_p&y_p\\
          d\bar y_p&\bar x_p\\
        \end{pmatrix}
        \begin{pmatrix}q_{11}&q_{12}\\
          dq_{21}'&q_{22}\\
        \end{pmatrix},
      \end{displaymath}
      where $x_p,y_p\in\Z_p[\myj]$ and $d'x_p\bar x_p-d''y_p\bar y_p\in\{-1,1\}$, and where the entries $q_{11}$, $q_{12}$, $q_{21}'$
      and $q_{22}$ are also in~$\Z_p[\myj]$, and
      satisfy $q_{11}q_{22}-dq_{21}'q_{21}\in\Z_p[\myj]^\times$. Comparing $(1,1)$ entries, we see that $x_1'+x_2'\myj\in d'\Z_p[\myj]$,
      and so $x_1',x_2'\in d'\Z_p\cap\Z=p^m\Z_p\cap\Z$, so that again $p^m$ divides~$x_1'$ and~$x_2'$.

      It follows that $d'$ divides $x_1'$ and~$x_2'$, and so we can write $x_1'=d'x_1$ and $x_2'=d'x_2$ with
      $x_1,x_2\in\Z$. Let $x=x_1+x_2\myj$ and $y=y_1+y_2\myj$. Then $u=x\bar x\in\cL$, $v=y\bar y\in\cL$ and
      $\myNrd(\xi)=\epsilon d'$ shows that $d'u-d''v=\epsilon$.
    \end{proof}
    \begin{corollary}Suppose that $d=p_1^{m_1}\cdots p_r^{m_r}$, with $p_1,\ldots,p_r$ distinct primes. There are
      $2^r$ Hall divisors $d'$ of~$d$, and for each such~$d'$ there are Atkin-Lehner elements~\eqref{eq:atkinlehner}.
    \end{corollary}
    \begin{subsection}{An example of Atkin-Lehner elements.}Let $\cH$ be as above for $d=40$.
      Then $\cH$ is the quaternion algebra over~$\Q$ with discriminant~10, and with canonical involution
      $x+y\phi\mapsto \bar x-y\phi$. The quaternion order~$\cO$ has discriminant~120. Since the level of~$\cO$ is not coprime
      to the discriminant of~$\cH$, this is not an Eichler order. In the B\'ezout relation $5\cdot21-8\cdot13=1$, we
      have $21=x\bar x$ and $13=y\bar y$ for $x=1+5\myj$ and $y=1-3\myj$. Let us define
      \begin{displaymath}
        w_5=5x+\bar y\phi\quad\text{and}\quad w_8=8y-x\phi.
      \end{displaymath}
      Then $\myNrd(w_5)=5$ and $\myNrd(w_8)=-8$, so in particular $w_5$ and~$w_8$ are not invertible in~$\cO$. These elements are however
      in the normalizer of~$\cO$ and define two Atkin-Lehner involutions on the Shimura curve $\mathcal{X}(\cO)$ associated to $\cO$. One
      has $w_5^2=5(-3w_5-1)$ and $w_8^2=8(5w_8+1)$, with $\myNrd(-3w_5-1)=\myNrd(5w_8+1)=1$. Moreover, $w_8w_5=-w_{40}$ for $w_{40}=\phi$.
      The element $w_3=1+2\myj$ is also in the normalizer of~$\cO$, with square $w_3^2=-3$.
    \end{subsection}
  \end{section}
\begin{section}{Elements of order~3.}\label{sec:order3elts}

In this Section, we give some preliminary results for the proof of Theorems~\ref{thm:bezoutresult} and~\ref{thm:Cdcount}.

   \begin{lemma}\label{lem:xiorder3conds}Suppose that $\xi=x_1+x_2\myj+x_3\phi+x_4\myj\phi\in\cO$ has order~3. Then
  \begin{equation}\label{eq:xiorder3equation}
    \xi^2+\xi+1=0.
  \end{equation}
  This holds if and only if $x_2=2x_1+1$ and
  \begin{equation}\label{eq:xiorder3cond}
    3x_1(x_1+1)=d(x_3^2-x_3x_4+x_4^2).
    \end{equation}
\end{lemma}
    \begin{proof}From $\xi^3=1$ we see that $\myNrd(\xi)=1$. Now $\myNrd(1-\xi)\ne0$, for otherwise we find that
      $x_2=2x_1-2$, and then equating $\xi^2$ and~$\xi^{-1}$ we find that $x_3=x_4=0$, and quickly get a contradiction.
      So $1-\xi$ is invertible, and so \eqref{eq:xiorder3equation}~holds. For any $\xi\in\cH$ we have $\xi^2-\myTr(\xi)\xi+\myNrd(\xi)=0$. So \eqref{eq:xiorder3equation}
      holds if and only if $-(\myTr(\xi)+1)\xi+\myNrd(\xi)-1=0$, which holds if and only if $\myTr(\xi)=-1$
    and $\myNrd(\xi)=1$.
    \end{proof}
    %
  \begin{lemma}\label{lem:xdashformula}If $\xi=x+y\phi\in\cH$, $\alpha=a+b\phi\in\cH$
    and $\myNrd(\alpha)=C\ne0$, then $\alpha\xi\alpha^{-1}=x'+y'\phi$ for
    $x'=x+dC^{-1}\bigl(b\bar b(x-\bar x)+\bar ab\bar y-a\bar by\bigr)$.
   In particular, if $\xi\in\cO$ and $\alpha\in\cO^\times$, then $x'-x\in d\mathfrak o$.
  \end{lemma}
    \begin{proof}This is a routine calculation, using $a\bar a-db\bar b=C$ and $\alpha^{-1}=C^{-1}(\bar a-b\phi)$.
    \end{proof}
  \begin{corollary}\label{cor:ddashddoubledashinvariant}Suppose that $\xi=x_1+x_2\myj+x_3\phi+x_4\myj\phi\in\cO$ has order~3.
    Let $d_\xi'=\gcd(x_1,d)$ and $d_\xi''=\gcd(x_1+1,d)$. Then $\gcd(d_\xi',d_\xi'')=1$. If $\eta=\alpha\xi\alpha^{-1}$
    for some $\alpha\in\cO^\times$, then $d_\eta'=d_\xi'$
    and $d_\eta''=d_\xi''$.
  \end{corollary}
  \begin{proof}Any common divisor $n\ge1$ of $d_\xi'$ and $d_\xi''$ divides both $x_1$ and $x_1+1$ and so is~1.
    If $\eta=\alpha\xi\alpha^{-1}$, write $\eta=y_1+y_2\myj+y_3\phi+y_4\myj\phi$. Then Lemma~\ref{lem:xdashformula} shows that $y_1=x_1+d\Delta$ for some integer~$\Delta$.
    So $\gcd(y_1,d)=\gcd(x_1,d)$, and $y_1+1=x_1+1+d\Delta$, shows that $\gcd(y_1+1,d)=\gcd(x_1+1,d)$ too.
  \end{proof}

    \begin{remark}\label{rem:effectofsquaring}(a)
    Suppose that $\xi=x_1+x_2\myj+x_3\phi+x_4\myj\phi\in\cO$ has order~3, and write $\xi^2=\eta=y_1+y_2\myj+y_3\phi+y_4\myj\phi$.
    Then $(d_\eta',d_\eta'')=(d_\xi'',d_\xi')$. This is immediate from~\eqref{eq:xiorder3equation},
    which implies that $y_1=-(x_1+1)$ (and $y_1+1=-x_1$).

    \noindent(b)
      Since $(d_{\xi^2}',d_{\xi^2}'')=(d_{\xi}'',d_{\xi}')$, for any subgroup of~$\cO^\times$ of order~3
      we can, if $d\ne1$, choose its generator~$\xi$ so that $d_\xi'<d_{\xi^2}'$. Then two subgroups are conjugate 
      if and only if their generators are conjugate. So
  the number of conjugation classes $C_d'$ under $\cO^\times$ of order~3 subgroups of $\cO^\times$ 
  is $C_d'=\frac{1}{2}C_d$ if $d\neq 1$, and $C_1'=1$.
  \end{remark}

    \begin{corollary}\label{cor:ddashddoubledashvalues}Suppose that $\xi=x_1+x_2\myj+x_3\phi+x_4\myj\phi\in\cO$ has order~3.
    Let $d_\xi'$ and $d_\xi''$ be as above. Then $d_\xi'd_\xi''$ equals either~$d$ or $d/3$. 
  \end{corollary}
    \begin{proof} Let $d^*=d/3$ if $d$ is divisible by~3, and otherwise let $d^*=d$. Then \eqref{eq:xiorder3cond} shows that $d^*$ divides~$x_1(x_1+1)$.
        Since $\gcd(x_1,x_1+1)=1$, this implies that $d^*=\gcd(x_1,d^*)\gcd(x_1+1,d^*)$, which divides $\gcd(x_1,d)\gcd(x_1+1,d)=d_\xi'd_\xi''$.
As $\gcd(d_\xi',d_\xi'')=1$ and both $d_\xi'$ and $d_\xi''$ divide~$d$, their product $d_\xi'd_\xi''$ divides~$d$.
  \end{proof}
We shall frequently use the following simple fact:
   \begin{lemma}\label{lem:oneminusomegafactor}Suppose that $\xi=x_1+x_2\myj+x_3\phi+x_4\myj\phi\in\cO$ has order~3
    and assume that
      $x_3^2-x_3x_4+x_4^2$ is divisible by~3.
    Then $x_3+x_4$ and $2x_3+x_4$ are also divisible by~3, so that
    \begin{equation}\label{eq:tildexformulas}
      \tilde x_3=\frac{2x_3-x_4}{3}\quad\text{and}\quad \tilde x_4=\frac{x_3+x_4}{3}
    \end{equation}
    are integers. Moreover,
    \begin{equation}\label{eq:tildexequation}
      d({\tilde x_3}^2-\tilde x_3\tilde x_4+{\tilde x_4}^2)=x_1(x_1+1).
    \end{equation}
  \end{lemma}
  \begin{proof}From $x_3^2-x_3x_4+x_4=(x_3+x_4)^2-3x_3x_4$, the hypothesis shows 
    that 3 divides~$x_3+x_4$, and therefore 3 also divides $3x_3-(x_3+x_4)=2x_3-x_4$. So
    \eqref{eq:tildexequation} follows from~\eqref{eq:xiorder3cond}.
  \end{proof}

  In studying whether or not the order $3$ elements $\xi=x_1+x_2\myj+x_3\phi+x_4\myj\phi\in\cO$ and $\eta=y_1+y_2\myj+y_3\phi+y_4\myj\phi\in\cO$ are
  conjugate by an $\alpha\in\cO^\times$, we shall repeatedly follow the same method, which we now summarize. Assume that $x_1+y_1+1\ne0$.
  For any $\alpha=a_1+a_2\myj+a_3\phi+a_4\myj\phi\in\cH$, $\eta\alpha=\alpha\xi$ if and only if
\begin{equation}\label{eq:a3a4generalformulas}
a_3=\frac{m_{31}a_1+m_{32}a_2}{3(x_1+y_1+1)}\quad\text{and}\quad
a_4=\frac{m_{41}a_1+m_{42}a_2}{3(x_1+y_1+1)},
      \end{equation}
      where
      \begin{displaymath}
        m_{31}=-x_3+2x_4+y_3-2y_4,\quad m_{32}=2x_3-x_4+y_3+y_4,
      \end{displaymath}
      and
      \begin{displaymath}
m_{41}=-2x_3+x_4+2y_3-y_4,\quad m_{42}=x_3+x_4-y_3+2y_4.
      \end{displaymath}
      The fact that $\eta\alpha=\alpha\xi$ implies that \eqref{eq:a3a4generalformulas} follows immediately from
      $x_2=2x_1+1$ and $y_2=2x_1+1$. The converse follows easily using~\eqref{eq:xiorder3cond}.

      We need to know when we can pick $\alpha\in\cO^\times$. Starting from any integers $a_1$ and~$a_2$, defining
      rational numbers $a_3$ and~$a_4$ using~\eqref{eq:a3a4generalformulas}, and forming $\alpha=a_1+a_2\myj+a_3\phi+a_4\myj\phi\in\cH$,
      we see that
      \begin{displaymath}
        \myNrd(\alpha)=a_1^2-a_1a_2+a_2^2-d(a_3^2-a_3a_4+a_4^2)=A_0a_1^2+B_0a_1a_2+C_0a_2^2,
      \end{displaymath}
      where
      \begin{displaymath}
        \begin{aligned}
 A_0&=(9S^2-d(m_{31}^2-m_{31}m_{41}+m_{41}^2))/(9S^2),\\
B_0&=-(9S^2+d(2m_{31}m_{32}-m_{31}m_{42}-m_{32}m_{41}+2m_{41}m_{42}))/(9S^2),\\
C_0&=(9S^2-d(m_{32}^2-m_{32}m_{42}+m_{42}^2))/(9S^2),\\
          \end{aligned}
      \end{displaymath}
      and $S=x_1+y_1+1$. A routine calculation shows that
      \begin{equation}\label{eq:discr0}
        B_0^2-4A_0C_0=\frac{-3}{(x_1+y_1+1)^2}.
      \end{equation}
      Looking at the first equation in~\eqref{eq:a3a4generalformulas}, we find that to ensure that $a_3$ is an integer,
      it is necessary and sufficient to have
      \begin{equation}\label{eq:umatrix}
                a_1=u_{11}m+u_{12}n,\quad\text{and}\quad
        a_2=u_{21}m+u_{22}n,
              \end{equation}
      for certain integer constants~$u_{ij}$ and arbitrary integers~$m$ and~$n$. Substituting
      the equations~\eqref{eq:umatrix} into $A_0a_1^2+B_0a_1a_2+C_0a_2^2$ we see that $\myNrd(\alpha)$
      is a binary quadratic form $A_1m^2+B_1mn+C_1n^2$ in~$m$ and~$n$, with discriminant
      \begin{displaymath}
        B_1^2-4A_1C_1=(B_0^2-4A_0C_0)(u_{11}u_{22}-u_{12}u_{21})^2=\frac{-3(u_{11}u_{22}-u_{12}u_{21})^2}{(x_1+y_1+1)^2}.
      \end{displaymath}
      Now looking at the second equation in~\eqref{eq:a3a4generalformulas}, we similarly find that to ensure that $a_4$ is an integer,
      it is necessary and sufficient to have
      \begin{equation}\label{eq:vmatrix}
        m=v_{11}k+v_{12}\ell,\quad\text{and}\quad n=v_{21}k+v_{22}\ell,
      \end{equation}
      for certain integer constants~$v_{ij}$ and arbitrary integers~$k$ and~$\ell$. Substituting
      the equations~\eqref{eq:vmatrix} into $A_1m^2+B_1mn+C_1n^2$ we see that $\myNrd(\alpha)$
      is a binary quadratic form $Ak^2+Bk\ell+C\ell^2$ in~$k$ and~$\ell$, with discriminant
      \begin{displaymath}
        B^2-4AC=(B_1^2-4A_1C_1)(v_{11}v_{22}-v_{12}v_{21})^2.
      \end{displaymath}
      So if we can show that  $(u_{11}u_{22}-u_{12}u_{21})(v_{11}v_{22}-v_{12}v_{21})=\pm(x_1+y_1+1)$,
      we have $B^2-4AC=-3$. Then $4AC=B^2+3$ shows that $A$ and~$C$ have the same sign. Let
      $\epsilon=1$ or~$-1$ according as $A,C>0$ or $A,C<0$. Then $\epsilon Ak^2+\epsilon Bk\ell+\epsilon C\ell^2$
      is a positive binary form having discriminant~$-3$. By \cite[p.~167]{nivenzuckerman}, there
      are integers $k$ and~$\ell$ so that $\epsilon(Ak^2+Bk\ell+C\ell^2)=1$. So defining $a_1,\ldots,a_4$ using
      these $k$ and~$\ell$, we have $\alpha\in\cO^\times$ and $\myNrd(\alpha)=\epsilon$.     
  
      If $\xi=x_1+x_2\myj+x_3\phi+x_4\myj\phi$ has order~3, let us define
        \begin{equation}\label{eq:xistardefn}
          \xi^*=\phi \xi^2\phi^{-1}.
        \end{equation}
        Then
 $\xi^*=x_1+x_2\myj+(x_4-x_3)\phi+x_4\myj\phi$ has order~3, and $(d_{\xi^*}',d_{\xi^*}'')=(d_\xi',d_\xi'')$.

   \begin{proposition}\label{prop:xistarconjugatetoxi}Let $\xi=x_1+x_2\myj+x_3\phi+x_4\myj\phi$ have order~3. Define $(d_\xi',d_\xi'')$
      as in Corollary~\ref{cor:ddashddoubledashinvariant} and $\xi^*$ as in~\eqref{eq:xistardefn}.
    Then the following are equivalent:
    \begin{itemize}
    \item[(a)] $d_\xi'd_\xi''=d$,
    \item[(b)] $x_3^2-x_3x_4+x_4^2$ is divisible by~3,
    \item[(c)] $\xi^*$ is $\cO^\times$-conjugate to~$\xi$.
    \end{itemize}
    These conditions always hold if $d$ is not divisible by~3 or if $d=3\tilde d$, where $\tilde d\equiv1$ (mod~3).
  \end{proposition}
  \begin{proof}Suppose that (a)~holds. Now $d_\xi'$ divides~$x_1$ and $d_\xi''$ divides~$x_1+1$. Writing $x_1=d_\xi'u$ and $x_1+1=d_\xi''v$,
    we have $d(x_3^2-x_3x_4+x_4^2)=3x_1(x_1+1)=3d_\xi'd_\xi''uv=3duv$. So $x_3^2-x_3x_4+x_4^2=3uv$ is divisible by~3.
    Conversely, if (b) holds, then (a) holds if $d$ is not divisible by~3. If $3^k\|d$ for some $k\ge1$, then
    (b) shows that $3^{k+1}$ divides~$d(x_3^2-x_3x_4+x_4^2)=3x_1(x_1+1)$, so that $3^k$ divides~$x_1$ (and so $d_\xi'$)
    or $3^k$ divides $x_1+1$ (and so $d_\xi''$). So $3^k$ divides $d_\xi'd_\xi''$, and therefore $d_\xi'd_\xi''$ cannot equal~$d/3$. So (a)~holds.

    By Corollary~\ref{cor:ddashddoubledashvalues}, $d_\xi'd_\xi''=d$ or $d/3$, and of course $d_\xi'd_\xi''=d$ must
    hold if $d$ is not divisible by~3. When $d=3\tilde d$ and $\tilde d=1$ (mod~3), then
    $x_1(x_1+1)=\tilde d(x_3^2-x_3x_4+x_4^2)=\tilde d((x_3+x_4)^2-3x_3x_4)$ is
    congruent to~0 (mod~3) if 3 divides~$x_3^2-x_3x_4+x_4^2$,
    and is congruent to~1 (mod~3) otherwise. The latter case is impossible,
    as $x_1(x_1+1)$ is congruent to~0 (mod~3) if $x_1\equiv0$ or $x_1\equiv2$ (mod~3), and
    congruent to~2 (mod~3) if $x_1\equiv1$ (mod~3). So $x_3^2-x_3x_4+x_4^2$ must be divisible by~3, and (b)~holds.

    Now suppose that (c) holds, but (b) does not hold. Then there an $\alpha=a_1+a_2\myj+a_3\phi+a_4\myj\phi\in\cO^\times$ such that
      $\xi^*\alpha=\alpha\xi$. Then \eqref{eq:a3a4generalformulas} becomes
      \begin{equation}\label{eq:a3formula}
    a_3=\frac{m_1a_1+m_2a_2}{3(2x_1+1)}\quad\text{and}\quad a_4=2a_3,
      \end{equation}
      where $m_1=x_4-2x_3$ and $m_2=x_3+x_4$.
      Now 3 cannnot divide $m_1$ or~$m_2$, because otherwise $x_3^2-x_3x_4+x_4^2=(x_3+x_4)^2-3x_3x_4$ is
      divisible by~3, so that (b)~holds. In particular, $m_1,m_2\ne0$.
      Let $g_3=\gcd(m_1,m_2)$, write $m_1=g_3m_1'$ and $m_2=g_3m_2'$, and
      choose integers $z$ and~$w$ so that $m_1z+m_2w=g_3$.
      
      Now $\gcd(g_3,3(2x_1+1))=1$. For if a positive integer~$n$ divides $m_1$ and~$m_2$, then
        $n$ is not divisible by~3, as we have seen, and $n$ divides $m_2-m_1=3x_3$ and $m_1+2m_2=3x_4$.
        Thus $n$ divides both $x_3$ and $x_4$ and so $d(x_3^2-x_3x_4+x_4^2)=3x_1(x_1+1)$.
        So $n$ divides $4x_1(x_1+1)=(2x_1+1)^2-1$, and so cannot divide~$2x_1+1$ unless $n=1$.
        This argument also shows that $g_3=\gcd(x_3,x_4)$.

        Now $3(2x_1+1)a_3=g_3(m_1'a_1+m_2'a_2)$ and $\gcd(g_3,3(2x_1+1))=1$ shows that
        $a_3$ must be a multiple $g_3m$ of~$g_3$.  If integers $a_1,a_2$ satisfy
        $3(2x_1+1)m=m_1'a_1+m_2'a_2$, then subtracting the
        equation $3(2x_1+1)m=m_1'(3(2x_1+1)zm)+m_2'(3(2x_1+1)wm)$, we obtain
        \begin{displaymath}
       m_1'(a_1-3(2x_1+1)zm)+m_2'(a_2-3(2x_1+1)wm)=0,
        \end{displaymath}
        and using $\gcd(m_1',m_2')=1$ we see that there must be an integer~$n$ so that
        \begin{equation}\label{eq:solvelineareqtn}
             a_1=3(2x_1+1)zm+m_2'n\quad\text{and}\quad
             a_2=3(2x_1+1)wm-m_1'n.
           \end{equation}
        Taking these $a_1$ and $a_2$ and setting $a_3=(m_1a_1+m_2a_2)/(3(2x_1+1))$ and $a_4=2a_3$,
        we see that $a_3=g_3m$, and so $\myNrd(\alpha)$ equals $Am^2+Bmn+Cn^2$ for
        certain integers~$A$, $B$ and~$C$. We find that
               \begin{displaymath}
          \begin{aligned}
            A&=9(2x_1+1)^2(z^2-zw+w^2)-3dg_3^2,\\
            B&=9(2x_1+1)((x_3-x_4)w+x_4z)/g_3,\\
            C&=3(x_3^2-x_3x_4+x_4^2)/g_3^2.
            \end{aligned}
               \end{displaymath}
              Clearly $A$ is divisible by~3, and
               $B$ and~$C$ are as well, because $g_3$ divides $x_3$ and~$x_4$. So $\myNrd(\alpha)$ must be a multiple of~3,
               and so $\alpha$ cannot be in~$\cO^\times$. This contradiction shows that (b) must hold.

               Now suppose that (b) holds. We show that (c)~holds. If $x_4-2x_3=0$, then $\xi^*=\xi$, and there
               is nothing to prove. So assume that $x_4-2x_3\ne0$. We seek $\alpha=a_1+a_2\myj+a_3\phi+a_4\myj\phi\in\cO^\times$ so that
               $\xi^*\alpha=\alpha\xi$. This means that \eqref{eq:a3formula} holds, and $\myNrd(\alpha)=\pm1$.
               Using Lemma~\ref{lem:oneminusomegafactor}, we see that \eqref{eq:a3formula} becomes
               \begin{equation}\label{eq:a3formulatilde}
    a_3=\frac{-\tilde x_3a_1+\tilde x_4a_2}{2x_1+1}\quad\text{and}\quad a_4=2a_3,
      \end{equation}
               where $\tilde x_3$ and $\tilde x_4$ are defined in~\eqref{eq:tildexformulas}.

               Let $\tilde g=\gcd(\tilde x_3,\tilde x_4)$. Then $\gcd(\tilde g,2x_1+1)=1$. For if a positive integer~$n$
               divides~$\tilde g$, then
               $n^2$ divides $4d({\tilde x_3}^2-\tilde x_3\tilde x_4+{\tilde x_4}^2)=4x_1(x_1+1)=(2x_1+1)^2-1$, and so $n$ cannot divide~$2x_1+1$
               unless $n=1$.
    So \eqref{eq:a3formulatilde} implies that $\tilde g$ divides~$a_3$, $a_3=\tilde gm$ say.
    Write $\tilde x_3=\tilde gx_3'$ and $\tilde x_4=\tilde gx_4'$. There are integers $z$ and~$w$ so
    that $\tilde x_3z+\tilde x_4w=\tilde g$. If $-\tilde x_3a_1+\tilde x_4a_2=\tilde gm(2x_1+1)$, then
     arguing as in the derivation of~\eqref{eq:solvelineareqtn},
     there must be an integer~$n$ so that
    \begin{displaymath}
      a_1=-mz(2x_1+1)+nx_4'\quad\text{and}\quad a_2=mw(2x_1+1)+nx_3'.
    \end{displaymath}
    This means that \eqref{eq:umatrix} holds for
    \begin{displaymath}
      \begin{pmatrix}
        u_{11}&u_{12}\\
        u_{21}&u_{22}
      \end{pmatrix}
      =\begin{pmatrix}
             -z(2x_1+1)&x_4'\\
             w(2x_1+1)&x_3'
      \end{pmatrix}.
    \end{displaymath}
    Now $u_{11}u_{22}-u_{12}u_{21}=-(2x_1+1)$, so $\myNrd(\alpha)$ is a quadratic form $A_1m^2+B_1mn+C_1n^2$
    in~$m$ and~$n$ of discriminant $-3(u_{11}u_{22}-u_{12}u_{21})^2/(2x_1+1)^2=-3$.
    By \cite[p.~167]{nivenzuckerman}, there
    are integers $m$ and~$n$ so that $A_1m^2+B_1mn+C_1n^2=\epsilon\in\{-1,1\}$. So defining
    $a_1,\ldots,a_4$ using these $m$ and~$n$,
    we have $\myNrd(\alpha)=\epsilon$, and so (c)~holds.  
  \end{proof}

  \begin{lemma}\label{lem:interchangex3x4}Let $\xi=x_1+x_2\myj+x_3\phi+x_4\myj\phi\in\cO$ have order~3.
    Then $\xi$ is conjugate to~$\eta=x_1+x_2\myj+x_4\phi+x_3j\phi$.
  \end{lemma}
  \begin{proof}There is nothing to prove if $x_3=x_4$, and so assume that $x_3\ne x_4$.
    We seek $\alpha\in\cO^\times$ such that $\eta\alpha=\alpha\xi$. In this case, \eqref{eq:a3a4generalformulas}
    becomes
    \begin{equation}\label{eq:a3formulainproof}
      a_3=\frac{a_1(x_4-x_3)+a_2x_3}{2x_1+1}\quad\text{and}\quad a_4=a_3.
    \end{equation}
    The condition $a_4=a_3$ implies that $\myNrd(\alpha)=a_1^2-a_1a_2+a_2^2-da_3^2$.

    If $x_3=0$, take $a_1=2x_1+1$, $a_2=x_1$ and $a_3=a_4=x_4$. Then \eqref{eq:a3formulainproof} holds,
  and $a_1^2-a_1a_2+a_2^2=3x_1(x_1+1)+1=dx_4^2+1=da_3^2+1$. So assume $x_3\ne0$ below.

    Let $g=\gcd(x_3,x_4)$. Then $\gcd(g,2x_1+1)=1$. For if a prime $p$ divides~$g$, then
    $p^2$ divides $d(x_3^2-x_3x_4+x_4^2)=3x_1(x_1+1)$, and so, even if $p=3$, $p$ divides
    $x_1(x_1+1)$. Since $\gcd(2x_1+1,x_1)=\gcd(2x_1+1,x_1+1)=1$, we get a contradiction.
    It follows that if \eqref{eq:a3formulainproof} holds, then $g$ divides~$a_3$, $a_3=gm$, say.
    There are integers $z$ and~$w$ so that $x_3z+x_4w=g$. Write $x_3=gx_3'$ and $x_4=gx_4'$.
    Then $w(x_4'-x_3')+(z+w)x_3'=1$, and so if $a_1(x_4-x_3)+a_2x_3=gm(2x_1+1)$, then
     arguing as in the derivation of~\eqref{eq:solvelineareqtn},
    there must be an integer~$n$ so that
    \begin{displaymath}
      a_1=mw(2x_1+1)+nx_3'\quad\text{and}\quad a_2=m(z+w)(2x_1+1)-n(x_4'-x_3'),
    \end{displaymath}
    and \eqref{eq:umatrix} holds for
    \begin{displaymath}
      \begin{pmatrix}
        u_{11}&u_{12}\\
        u_{21}&u_{22}
      \end{pmatrix}
      =
           \begin{pmatrix}
             w(2x_1+1)&x_3'\\
             (z+w)(2x_1+1)&x_3'-x_4'\\
      \end{pmatrix}
    \end{displaymath}
    Note that $u_{11}u_{22}-u_{12}u_{21}=-(2x_1+1)$. Defining $a_1$ and~$a_2$ using~\eqref{eq:umatrix}
    for these~$u_{ij}$'s, we see that $a_3$ and~$a_4$ are integer linear combinations of~$m$ and~$n$,
    and so $\myNrd(\alpha)=A_1m^2+B_1mn+C_1n^2$ for integers $A_1$, $B_1$ and $C_1$ satisfying
    \begin{displaymath}
      B_1^2-4A_1C_1=\frac{-3(u_{11}u_{22}-u_{12}u_{21})^2}{(2x_1+1)^2}=-3,
    \end{displaymath}
    and so as usual, $m$ and $n$ may be chosen so that $\alpha\in\cO^\times$.
  \end{proof}

  \begin{proposition}\label{prop:mainconjugacyresult}Suppose that $\xi,\eta\in\cO$ have order~3, and 
    that $(d_\xi',d_\xi'')=(d_\eta',d_\eta'')$. Then $\eta$ is $\cO^\times$-conjugate to either $\xi$ or~$\xi^*=\phi\xi^2\phi^{-1}$.
  \end{proposition}
  \begin{proof}Write $\xi=x_1+x_2\myj+x_3\phi+x_4\myj\phi$, $\eta=y_1+y_2\myj+y_3\phi+y_4\myj\phi$, 
    $d'=d_\xi'=d_\eta'$ and $d''=d_\xi''=d_\eta''$. We treat the cases $d'd''=d$ and $d'd''=d/3$ separately.
    
    \smallskip\noindent{\bf 1. The case $\boldsymbol{d'd''=d}$.} Proposition~\ref{prop:xistarconjugatetoxi} shows that $\xi^*$ is conjugate to~$\xi$, and so
    we must show that $\eta$ is conjugate to~$\xi$. Proposition~\ref{prop:xistarconjugatetoxi} also shows that 3 divides $x_3^2-x_3x_4+x_4^2$
    and~$y_3^2-y_3y_4+y_4^2$.

    When $d=1$, $d'=d''=1$ must hold, and $(d_\xi',d_\xi'')=(d_\eta',d_\eta'')$ is automatically satisfied. In that
    case we'll show that $\xi$ and~$\eta$ are conjugate by showing that any $\eta$ of order~3 is conjugate to~$\myj$.
    So in the proof below, for $d=1$, take $\xi=\myj$.

    Applying Lemma~\ref{lem:oneminusomegafactor} to both $\xi$ and~$\eta$, we have integers
  \begin{displaymath}
      \tilde x_3=\frac{2x_3-x_4}{3},\quad \tilde x_4=\frac{x_3+x_4}{3},\quad \tilde y_3=\frac{2y_3-y_4}{3}\quad\text{and}\quad \tilde y_4=\frac{y_3+y_4}{3},
  \end{displaymath}
  for which
  \begin{equation}\label{eq:newtildexyformulas}
d({\tilde x_3}^2-\tilde x_3\tilde x_4+{\tilde x_4}^2)=x_1(x_1+1)\quad\text{and}\quad d({\tilde y_3}^2-\tilde y_3\tilde y_4+{\tilde y_4}^2)=y_1(y_1+1).
  \end{equation}
  Write $x_1=d'u'$ and $x_1+1=d''u''$, and similarly $y_1=d'v'$ and $y_1+1=d''v''$. Then $d'$ divides $d'v'-d'u'=y_1-x_1$,
  and $d''$ divides $d''v''-d''u''=(y_1+1)-(x_1+1)=y_1-x_1$. Since $\gcd(d',d'')=1$, $d=d'd''$ divides $y_1-x_1$. Write
  $y_1=x_1+d\Delta$.

  If $x_1+y_1+1=0$, then $d''$ divides $x_1=-(y_1+1)$, and so $d$ divides~$x_1$. Similarly, $d$ divides $x_1+1$, and therefore
  $d$ must be~1. Recall that in the case $d=1$ we are taking $\xi=\myj$, so that $x_1=0$. Now $y_1+1=0$ and so \eqref{eq:newtildexyformulas}
  shows that $\tilde y_3=\tilde y_4=0$. Thus $y_3=y_4=0$ and $\eta=\myj^2=\phi\myj\phi^{-1}$. As $\phi\in\cO^\times$ when $d=1$,
  $\eta$ is conjugate to~$\myj$.

  So we may assume below that $x_1+y_1+1\ne0$.
  
  If $\alpha=a_1+a_2\myj+a_3\phi+a_4\myj\phi\in\cO^\times$, then the equations~\eqref{eq:a3a4generalformulas}, which
  are necessary and sufficient for $\eta\alpha=\alpha\xi$ to hold, take the form
      \begin{equation}\label{eq:newa3a4formula}
    a_3=\frac{n_{31}a_1+n_{32}a_2}{x_1+y_1+1}\quad\text{and}\quad a_4=\frac{n_{41}a_1+n_{42}a_2}{x_1+y_1+1},
  \end{equation}
where
  \begin{displaymath}
    n_{31}=\tilde x_4-\tilde x_3+\tilde y_3-\tilde y_4,\ 
    n_{32}=\tilde x_3+\tilde y_4,\ n_{41}=\tilde y_3-\tilde x_3\ \text{and}\  n_{42}=\tilde x_4-\tilde y_3+\tilde y_4.
  \end{displaymath}
  Notice that
  \begin{equation}\label{eq:n41n32identity}
    \begin{aligned}
      n_{41}n_{32}-n_{42}n_{31}&=({\tilde y_3}^2-\tilde y_3\tilde y_4+{\tilde y_4}^2)-({\tilde x_3}^2-\tilde x_3\tilde x_4+{\tilde x_4}^2)\\
      &=y_1(y_1+1)/d-x_1(x_1+1)/d\\
      &=(y_1-x_1)(x_1+y_1+1)/d\\
           &=(x_1+y_1+1)\Delta.\\
      \end{aligned}
  \end{equation}
If $n_{31}$, $n_{32}$, $n_{41}$ and $n_{42}$ are all zero, then we find that $\tilde x_3=\tilde x_4=\tilde y_3=\tilde y_4=0$,
and so $x_1(x_1+1)=0=y_1(y_1+1)$ by~\eqref{eq:newtildexyformulas}. As $x_1+y_1+1\ne0$, we have either $\xi=\eta=\myj$ or $\xi=\eta=\myj^2$.
So we can assume that either $(n_{31},n_{32})\ne(0,0)$ or $(n_{41},n_{42})\ne(0,0)$. Both hold if $\Delta\ne0$, as we see from~\eqref{eq:n41n32identity},
but we need to also consider the case $\Delta=0$.

Let's first assume that $(n_{31},n_{32})\ne(0,0)$. Then we can form $g_3=\gcd(n_{31},n_{32})$, and write $n_{31}=g_3n_{31}'$ and $n_{32}=g_3n_{32}'$. As $\gcd(g_3',g_4')=1$,
we can choose $z,w\in\Z$ so that $n_{31}'z+wn_{32}'z=1$. Let $g_3''=\gcd(g_3,x_1+y_1+1)$, and
write $g_3=g_3'g_3''$ and $x_1+y_1+1=g_3''\delta_3$. Then $\gcd(g_3',\delta_3)=1$, and 
the first equation in~\eqref{eq:newa3a4formula} becomes
\begin{displaymath}
    a_3=\frac{g_3(n_{31}'a_1+n_{32}'a_2)}{g_3''\delta_3}=\frac{g_3'(n_{31}'a_1+n_{32}'a_2)}{\delta_3}.
\end{displaymath}
From $\delta_3a_3=g_3'(n_{31}'a_1+n_{32}'a_2)$ and $\gcd(g_3',\delta_3)=1$, we see that
$a_3$ must be a multiple $g_3'm$ of~$g_3'$, and $m\delta_3=n_{31}'a_1+n_{32}'a_2$. So arguing as
in the derivation of~\eqref{eq:solvelineareqtn}, there must be an integer~$n$ so that
\begin{equation}\label{eq:a1a2mnformulas}
  a_1=u_{11}m+u_{12}n\quad \text{and}\quad a_2=u_{21}m+u_{22}n
\end{equation}
for
\begin{displaymath}
  \begin{pmatrix}
    u_{11}&u_{12}\\
    u_{21}&u_{22}\\
  \end{pmatrix}=
    \begin{pmatrix}
      z\delta_3&n_{32}'\\
      w\delta_3&-n_{31}'
  \end{pmatrix}.
\end{displaymath}
Notice that the determinant of this matrix is~$-\delta_3$.
Substituting these into~\eqref{eq:newa3a4formula} we get
$a_3=g_3'm(zn_{31}'+wn_{32}')=g_3'm$ and
\begin{equation}\label{eq:a4deltadashformula}
  a_4=\frac{m\delta_3(n_{41}z+n_{42}w)}{x_1+y_1+1}+\frac{(n_{31}n_{42}'-n_{32}n_{41}')n}{x_1+y_1+1}
    =\frac{m(n_{41}z+n_{42}w)}{g_3''}+\frac{\Delta n}{g_3}.      
\end{equation}
Let us next consider the special case when $(n_{31},n_{32})\ne(0,0)$, but $(n_{41},n_{42})=(0,0)$.
Then $\Delta=0$ by~\eqref{eq:n41n32identity} because $x_1+y_1+1\ne0$, and taking $a_1$ and~$a_2$ as in~\eqref{eq:a1a2mnformulas},
we have $a_3=g_3'm$ and $a_4=0$. So $a_1$, $a_2$, $a_3$ and~$a_4$ are all integer linear combinations
of $m$ and~$n$, and $\myNrd(\alpha)$ equals $A_1m^2+B_1mn+C_1n^2$ for integers $A_1$, $B_1$ and~$C_1$ satisfying
\begin{displaymath}
  B_1^2-4A_1C_1=\frac{-3(u_{11}u_{22}-u_{12}u_{21})^2}{(x_1+y_1+1)^2}=\frac{-3\delta_3^2}{(x_1+y_1+1)^2}=\frac{-3}{{g_3''}^2}.
\end{displaymath}
As $A_1,B_1,C_1\in\Z$, $g_3''$ must be~1, and  $B_1^2-4A_1C_1$ is in fact equal to~$-3$, and so in the usual way
we see that in this case there is an $\alpha\in\cO^\times$ such that $\eta\alpha=\alpha\xi$.

The proof when $(n_{31},n_{32})=(0,0)$ and $(n_{41},n_{42})\ne(0,0)$ is analogous.

So now assume that $(n_{31},n_{32})\ne(0,0)$ and $(n_{41},n_{42})\ne(0,0)$. Let $g_3$, $g_3'$,
$g_3''$ and~$\delta_3$ be as above, and now also
form $g_4=\gcd(n_{41},n_{42})$, write $n_{41}=g_4n_{41}'$, $n_{42}=g_4n_{42}'$, form
$g_4''=\gcd(g_4,x_1+y_1+1)$, and write $g_4=g_4'g_4''$ and $x_1+y_1+1=g_4''\delta_4$, with
$\gcd(n_{41}',n_{42}')=1=\gcd(g_4',\delta_4)$.

Now $\Delta$ is divisible by~$g_3'$. For by~\eqref{eq:n41n32identity}, we can
write $(x_1+y_1+1)\Delta=g_3g_4n$ for some integer~$n$. 
So $g_3'g_3''g_4n=g_3''\delta_3\Delta$, and $g_3'g_4n=\delta_3\Delta$. Since
$\gcd(g_3',\delta_3)=1$, we see that $g_3'$ divides~$\Delta$.

Writing $\Delta=g_3'\Delta'$, \eqref{eq:a4deltadashformula} becomes
  \begin{equation}\label{eq:a4deltadashformula2}
    a_4=\frac{(n_{41}z+n_{42}w)m+\Delta'n}{g_3''}.
  \end{equation}
  At least one of $n_{41}z+n_{42}w$ and $\Delta'$ is nonzero. For if both are zero,
  then $\Delta=0$ and
  \begin{displaymath}
    \begin{aligned}
      0=(n_{41}n_{32}-n_{42}n_{31})z&=(n_{41}z)n_{32}-n_{42}n_{31}z\\
      &=(-n_{42}w)n_{32}-n_{42}n_{31}z=-n_{42}(n_{31}z+n_{32}w)=-n_{42}g_3,
      \end{aligned}
  \end{displaymath}
  so that $n_{42}=0$. Similarly, $n_{41}=0$, and so the assumption $(n_{41},n_{42})\ne(0,0)$ is contradicted.
  
  So we can form $h=\gcd(n_{41}z+n_{42}w,\Delta')$, and write
  $n_{41}z+n_{42}w=c_1h$, $\Delta'=c_2h$, where $\gcd(c_1,c_2)=1$. Let $h''=\gcd(h,g_3'')$ and write
  $g_3''=h''g_3^*$ and $h=h'h''$, where $\gcd(h',g_3^*)=1$.
  Then \eqref{eq:a4deltadashformula2} becomes
  \begin{equation}\label{eq:a4deltadashformula3}
    a_4=\frac{h(c_1m+c_2n)}{g_3''}=\frac{h'(c_1m+c_2n)}{g_3^*}.
  \end{equation}
  Since $\gcd(h',g_3^*)=1$, we see that $a_4$ must be a multiple $h'k$ of~$h'$. Choose integers $u$, $v$ so that
  $c_1u+c_2v=1$. Arguing as in the derivation of~\eqref{eq:solvelineareqtn}, we see that
  the equation $g_3^*k=c_1m+c_2n$ holds if and only if
  \begin{equation}\label{eq:mandnintermsofkandell}
m=v_{11}k+v_{12}\ell\quad\text{and}\quad n=v_{21}k+v_{22}\ell
  \end{equation}
  for some integer~$\ell$, where
  \begin{displaymath}
    \begin{pmatrix}
      v_{11}&v_{12}\\
      v_{21}&v_{22}
    \end{pmatrix}
    =\begin{pmatrix}
    ug_3^*&c_2\\
    vg_3^*&-c_1\\
    \end{pmatrix}.
  \end{displaymath}
  Note that $v_{11}v_{22}-v_{12}v_{21}=-g_3^*$. Conversely, if we start from $k,\ell\in\Z$ and define $m$ and~$n$ by~\eqref{eq:mandnintermsofkandell},
  and $a_1$ and~$a_2$ by~\eqref{eq:a1a2mnformulas},
  then $a_1$, $a_2$, $a_3$ and~$a_4$ are all integer linear combinations of~$k$ and~$\ell$.
 So $\myNrd(\alpha)=Ak^2+Bk\ell+C\ell^2$ for some integers $A$, $B$ and~$C$
 satisfying
 \begin{displaymath}
        B^2-4AC=(B_1^2-4A_1C_1)(v_{11}v_{22}-v_{12}v_{21})^2=\frac{-3{g_3^*}^2}{{g_3''}^2}=\frac{-3}{{h''}^2}.
      \end{displaymath}
    As $A,B,C\in\Z$, the positive integer $h''$ must equal~1. So
    $B^2-4AC=-3$, and in the usual way
we see that in this case there is an $\alpha\in\cO^\times$ such that $\eta\alpha=\alpha\xi$.

\medskip
{\bf 2. The case $\boldsymbol{d'd''=d/3}$.}
      Write $\tilde d=d/3$ and $\xi^*=x_1^*+x_2^*j+x_3^*\phi+x_4^*j\phi$,
      where $x_1^*=x_1$, $x_2^*=x_2$, $x_3^*=x_4-x_3$ and $x_4^*=x_4$. Then
      ${x_3^*}^2-x_3^*x_4^*+{x_4^*}^2=x_3^2-x_3x_4+x_4^2$. By Proposition~\ref{prop:xistarconjugatetoxi},
      neither $x_3^2-x_3x_4+x_4^2$ nor $y_3^2-y_3y_4+y_4^2$ is divisible by~3.

      Suppose that $x_1+y_1+1=0$. Then $d''$ divides $y_1+1=-x_1$, so that $\tilde d$ divides~$x_1$.
      Also, $d'$ divides $-y_1=x_1+1$, so that $\tilde d$ divides~$x_1+1$. Hence $\tilde d=1$, which
      contradicts the last sentence in the statement of Proposition~\ref{prop:xistarconjugatetoxi}. So $x_1+y_1+1\ne0$.

      Write $x_1=d'u'$, $x_1+1=d''u''$, $y_1=d'v'$ and $y_1+1=d''v''$. Then
      $d'$ divides $d'v'-d'u'=y_1-x_1$ and 
      $d''$ divides $d''v''-d''u''=(y_1+1)-(x_1+1)=y_1-x_1$. Hence $\tilde d$ divides
      $y_1-x_1$. Suppose that $3^k\|d$. Then $k\ge1$. If $k=1$, then $\tilde d(x_3^2-x_3x_4+x_4^2)=x_1(x_1+1)$
      shows that 3 does not divide either $x_1$ or~$x_1+1$, and so $x_1\equiv1$ (mod~3) must hold.
      Similarly $y_1\equiv1$ (mod~3) and so 3 divides~$y_1-x_1$. As $\gcd(3,\tilde d)=1$ in this $k=1$ case,
      we see that $d$ divides~$y_1-x_1$. When $k\ge2$, write $d=3^kd^*$. Then $3^{k-1}$ divides $d/3=d'd''$, and
      so $3^{k-1}$ divides $d'$ or~$d''$. Assuming $3^{k-1}$ divides~$d'$, then it divides $x_1$ and~$y_1$. Writing
      $x_1=3^{k-1}x_1'$ and $y_1=3^{k-1}y_1'$, we have $d^*(x_3^2-x_3x_4+x_4^2)=x_1'(3^{k-1}x_1'+1)$, and so
      $x_1'\equiv d^*$ (mod~3). Similarly $y_1'\equiv d^*$. Hence $y_1-x_1=3^{k-1}(y_1'-x_1')$ is divisible by~$3^k$,
      and hence by~$d$. So we can write $y_1=x_1+\Delta d$ for an integer~$\Delta$.

      Let $\alpha=a_1+a_2\myj+a_3j+a_4\myj\phi\in\cO$. Then $\eta\alpha=\alpha\xi$  if and only if the equations~\eqref{eq:a3a4generalformulas}
      hold, and similarly, $\eta\alpha=\alpha\xi^*$  if and only if
      \begin{displaymath}
a_3=\frac{h_{31}a_1+h_{32}a_2}{3(x_1+y_1+1)}\quad\text{and}\quad
a_4=\frac{h_{41}a_1+h_{42}a_2}{3(x_1+y_1+1)},
      \end{displaymath}
      where
      \begin{displaymath}
        h_{31}=x_3+x_4+y_3-2y_4,\quad h_{32}=-2x_3+x_4+y_3+y_4,
      \end{displaymath}
      and
      \begin{displaymath}
h_{41}=2x_3-x_4+2y_3-y_4,\quad h_{42}=-x_3+2x_4-y_3+2y_4.
      \end{displaymath}
      Let us next show that either 3 divides all four $m_{ij}$'s and none of the $h_{ij}$'s, or vice versa.
      Firstly, $m_{31}-m_{32}=3(-x_3+x_4-y_4)$ and $h_{31}-h_{32}=3(x_3-y_4)$, and so
      $m_{31}\equiv m_{32}$ and $h_{31}\equiv h_{32}$ (mod~3). Suppose that $m_{31}=a+3k$, $m_{32}=a+3\ell$
      and $h_{31}=b+3k'$, $h_{32}=b+3\ell'$, where $a,b\in\{0,1,2\}$. Then we find that
      $x_3+x_4=a-b+3(k-k'+x_3)$. But 3 does not divide $x_3+x_4$, and so $a\ne b$. Also,
      $y_3+y_4=2a-b+3(\ell+k-k')$ is not divisible by~3, and so $(a,b)\ne(1,2)$ and $(b,a)\ne(2,1)$.
      So $a=0$ and $b\ne0$ or vice versa. Moreover, we calculate that
      $m_{41}\equiv -a$ and $m_{42}\equiv -a$ (mod~3) and $h_{41}\equiv -b$ and $h_{42}\equiv -b$ (mod~3).
      
        Notice that
  \begin{equation}\label{eq:m41m32identity}
    \begin{aligned}
      m_{41}m_{32}-m_{42}m_{31}&=3\bigl((y_3^2-y_3y_4+y_4^2)-(x_3^2-x_3x_4+x_4^2)\bigr)\\
      &=9y_1(y_1+1)/d-9x_1(x_1+1)/d\\
      &=9(y_1-x_1)(x_1+y_1+1)/d\\
           &=9(x_1+y_1+1)\Delta.\\
      \end{aligned}
  \end{equation}
       If $(m_{31},m_{32})=(0,0)$, then $y_3=-x_3$ and $y_4=-x_3+x_4$, so $y_3^2-y_3y_4+y_4^2=x_3^2-x_3x_4+x_4^2$,
      and therefore $x_1(x_1+1)=y_1(y_1+1)$. Then $(x_1+y_1+1)(y_1-x_1)=0$, so that $y_1=x_1$. So
      $(y_1,y_2,y_3,y_4)=(x_1,x_2,-x_3,-x_3+x_4)$, and $\eta$ is conjugate to
      $\myj\xi\myj^{-1}$ by Lemma~\ref{lem:interchangex3x4}, and so $\xi$ and~$\eta$ are conjugate.
      
      Similarly, if $(m_{41},m_{42})=(0,0)$, then $y_3=x_3-x_4$ and $y_4=-x_4$. Again
      $y_3^2-y_3y_4+y_4^2=x_3^2-x_3x_4+x_4^2$ and $x_1(x_1+1)=y_1(y_1+1)$.
      So $(y_1,y_2,y_3,y_4)=(x_1,x_2,x_3-x_4,-x_4)$, and $\eta$ is conjugate to
      $\myj^{-1}\xi\myj$ by Lemma~\ref{lem:interchangex3x4}, and again $\xi$ and~$\eta$ are conjugate.

      So we may assume below that $(m_{31},m_{32})\ne(0,0)$ and $(m_{41},m_{42})\ne(0,0)$.
      Form $g_3=\gcd(m_{31},m_{32})$ and $g_4=\gcd(m_{41},m_{42})$. 
     
      Interchanging the roles of $\xi$ and~$\xi^*$ if necessary, we may assume that 3 divides all four $m_{ij}$'s
      and none of the $h_{ij}$'s. We claim that $\eta$ is conjugate to~$\xi$.

      Let $g_3''=\gcd(g_3,3(x_1+y_1+1))$, and write $g_3=g_3'g_3''$ and $3(x_1+y_1+1)=g_3''\delta_3$.
      Then $\gcd(g_3',\delta_3)=1$. The first formula in~\eqref{eq:a3a4generalformulas} becomes
      \begin{equation}\label{eq:a3delta3eq}
a_3=\frac{g_3'(m_{31}'a_1+m_{32}'a_2)}{\delta_3}.
      \end{equation}
      Since $\gcd(g_3',\delta_3)=1$, any integer $a_3$ for which this holds must be a multiple $g_3'm$
      of~$g_3'$. Choose integers $z,w$ so that $m_{31}z+m_{32}w=g_3$.  Arguing as in the derivation of~\eqref{eq:solvelineareqtn},
        we see that the $(a_1,a_2)$ for which
      \eqref{eq:a3delta3eq} holds for $a_3=g_3'm$ are
      \begin{equation}\label{eq:a1a2mnformulascase2}
  a_1=u_{11}m+u_{12}n\quad \text{and}\quad a_2=u_{21}m+u_{22}n
\end{equation}
for
\begin{displaymath}
  \begin{pmatrix}
    u_{11}&u_{12}\\
    u_{21}&u_{22}\\
  \end{pmatrix}=
    \begin{pmatrix}
      z\delta_3&m_{32}'\\
      w\delta_3&-m_{31}'
  \end{pmatrix}.
\end{displaymath}
Notice that the determinant of this matrix is~$-\delta_3$.
 Substituting these values of~$a_1$ and~$a_2$ into~\eqref{eq:a3a4generalformulas}, we get
      $a_3=g_3'm$, and
      \begin{displaymath}
        \begin{aligned}
          a_4&=\frac{\delta_3(m_{41}z+m_{42}w)m}{3(x_1+y_1+1)}+\frac{(m_{41}m_{32}-m_{42}m_{31})n}{3g_3(x_1+y_1+1)}\\
          &=\frac{(m_{41}z+m_{42}w)m}{g_3''}+\frac{3\Delta n}{g_3}.\\
          \end{aligned}
        \end{displaymath}
      using~\eqref{eq:m41m32identity}. From~\eqref{eq:m41m32identity}, we also see that
      $g_3'g_3''g_4=g_3g_4$ divides $9(x_1+y_1+1)\Delta=3g_3''\delta_3\Delta$, and so
      $g_3'$ divides~$3\Delta$, as $\gcd(g_3',\delta_3)=1$. Let $\Delta'=3\Delta/g_3'$. So our equation for~$a_4$ becomes
           \begin{equation}\label{eq:a4formulawithDeltadash}
          a_4=\frac{(m_{41}z+m_{42}w)m+\Delta' n}{g_3''}.
           \end{equation}
           At least one of the numbers $\Delta'$ and $m_{41}z+m_{42}w$ is non zero. For otherwise,
           $\Delta=0$, and \eqref{eq:m41m32identity} shows that
  \begin{displaymath}
    \begin{aligned}
      0=(m_{41}m_{32}-m_{42}m_{31})z&=(m_{41}z)m_{32}-m_{42}m_{31}z\\
      &=(-m_{42}w)m_{32}-m_{42}m_{31}z\\
      &=-m_{42}(m_{31}z+m_{32}w)=-m_{42}g_3,
      \end{aligned}
  \end{displaymath}
  so that $m_{42}=0$. Similarly, $m_{41}=0$, which contradicts the assumption  we are making
  that $(m_{41},m_{42})\ne(0,0)$.

  So we can form $h=\gcd(m_{41}z+m_{42}w,\Delta')$, and write
  $m_{41}z+m_{42}w=c_1h$, $\Delta'=c_2h$, where $\gcd(c_1,c_2)=1$. Let $h''=\gcd(h,g_3'')$ and write
  $g_3''=h''g_3^*$ and $h=h'h''$, where $\gcd(h',g_3^*)=1$.
  Then \eqref{eq:a4formulawithDeltadash} becomes
  \begin{displaymath}
    a_4=\frac{h(c_1m+c_2n)}{g_3''}=\frac{h'(c_1m+c_2n)}{g_3^*}.
  \end{displaymath}
  Since $\gcd(h',g_3^*)=1$, we see that $a_4$ must be a multiple $h'k$ of~$h'$. Choose integers $u$, $v$ so that
  $c_1u+c_2v=1$. Arguing as in the derivation of~\eqref{eq:solvelineareqtn}, we see that
  the equation $g_3^*k=c_1m+c_2n$ holds if and only if
  \begin{equation}\label{eq:mandnintermsofkandellcase2}
m=v_{11}k+v_{12}\ell\quad\text{and}\quad n=v_{21}k+v_{22}\ell
  \end{equation}
  for some integer~$\ell$, where
  \begin{displaymath}
    \begin{pmatrix}
      v_{11}&v_{12}\\
      v_{21}&v_{22}
    \end{pmatrix}
    =\begin{pmatrix}
    ug_3^*&c_2\\
    vg_3^*&-c_1\\
    \end{pmatrix}.
  \end{displaymath}
  Note that $v_{11}v_{22}-v_{12}v_{21}=-g_3^*$. Conversely, if we start from $k,\ell\in\Z$ and define $m$ and~$n$ by~\eqref{eq:mandnintermsofkandellcase2},
  and $a_1$ and~$a_2$ by~\eqref{eq:a1a2mnformulascase2}
  then $a_1$, $a_2$, $a_3$ and~$a_4$ are all integer linear combinations of~$k$ and~$\ell$, so
  $\myNrd(\alpha)=Ak^2+Bk\ell+C\ell^2$ for some
  integers $A$, $B$ and~$C$ for which
  \begin{displaymath}
    \begin{aligned}
      B^2-4AC&=(B_1^2-4A_1C_1)(v_{11}v_{22}-v_{12}v_{21})^2\\
      &=(B_0^2-4A_0C_0)(u_{11}u_{22}-u_{12}u_{21})^2(v_{11}v_{22}-v_{12}v_{21})^2\\
      &=\frac{-3{g_3^*}^2\delta_3^2}{(x_1+y_1+1)^2}\\
      &=\frac{-27}{{h''}^2},
      \end{aligned}
  \end{displaymath}
  by~\eqref{eq:discr0}. Since the $B^2-4AC$ is an integer, $h''$ must be~1 or~3. To see that $h''=3$, suppose
  that $3^{\gamma_3}\| g_3$ and $3^{\gamma_4}\|g_4$. Then $\gamma_3,\gamma_4\ge1$ because we are assuming that
  3 divides all the $m_{ij}$'s.  If $\Delta=0$, then $\Delta'=0$ and $h=m_{41}z+m_{42}w$, which is divisible by~3.
  As 3 also divides $g_3''=\gcd(g_3,3(x_1+y_1+1))$, it divides $h''=\gcd(h,g_3'')$.
  So assume $\Delta\ne0$, and that $3^\beta\|3(x_1+y_1+1)$ and $3^\gamma\|\Delta$.
  If $\gamma_3\le\beta$, then $3^{\gamma_3}\|\gcd(g_3,3(x_1+y_1+1))=g_3''$, and so 3 does not divide $g_3'=g_3/g_3''$.
  So 3 divides $\Delta'=3\Delta/g_3'$, and therefore 3 divides $h=\gcd(g_{41}z+g_{42}w,\Delta')$, and
  $h''=\gcd(h,g_3'')$. If $\gamma_3>\beta$,
  then $3^\beta\|g_3''$ and so $3^{\gamma_3-\beta}\| g_3'$. From~\eqref{eq:m41m32identity} we see that $\beta+\gamma+1\ge\gamma_3+\gamma_4$,
  and so $\gamma\ge\gamma_3-\beta+\gamma_4-1\ge\gamma_3-\beta$. But $3^\gamma\|\Delta$ and $3^{\gamma_3-\beta}\|g_3'$, so
  $g_3'$ divides $\Delta$ because it divides~$3\Delta$. Hence 3 again divides $\Delta'$, and therefore 3 divides~$h''$, as before.  

  Hence $B^2-4AC=-3$, and so in the usual way we see that there is an $\alpha\in\cO^\times$ such that $\eta\alpha=\alpha\xi$.
  \end{proof}
  \begin{subsection}{Action of Atkin-Lehner elements on~$\cO$}
     For any $\xi=x_1+x_2\myj+x_3\phi+x_4\myj\phi\in\cO$, let $d_\xi'=\gcd(d,x_1)$. 
      \begin{lemma}\label{lem:atkinlehneraction}Let $e'$ be a Hall divisor of~$d$,
        and write $e''=d/e'$. Let $z,t\in\mathfrak o$ satisfy $e'z\bar z-e''t\bar t=\epsilon\in\{-1,1\}$, and let $w_{e'}=e'z+t\phi$
        be the corresponding Atkin-Lehner element of~$\cO$. Let $\xi=x_1+x_2\myj+x_3\phi+x_4\myj\phi\in\cO$, and let $\xi'=w_{e'}\xi w_{e'}^{-1}$. Then
        \begin{equation}\label{eq:dashxidash}
          d_{\xi'}'=\gcd(e',x_1-x_2)\gcd(e'',x_1).
        \end{equation}
        In particular, if $\xi$ has order~3, then $d_{\xi'}'=\gcd(e',x_1+1)\gcd(e'',x_1)$.
      \end{lemma}
      \begin{proof}Write $x=x_1+x_2\myj$, $y=x_3+x_4\myj$ and $\xi'=x'+y'\phi=x_1'+x_2'\myj+x_3'\phi+x_4'\myj\phi$.
        From $w_{e'}^{-1}=(\epsilon/e')(e'\bar z-t\phi)$, we see from Lemma~\ref{lem:xdashformula} that 
        \begin{displaymath}
          x'=\epsilon(e'xz\bar z-e''\bar xt\bar t)+ud=\epsilon(e'x_1z\bar z-e''(x_1-x_2)t\bar t)+m\myj+ud,
        \end{displaymath}
        for some $u\in\mathfrak o$ and $m\in\Z$, so that $x_1'=\epsilon(e'x_1z\bar z-e''(x_1-x_2)t\bar t)$ (mod~$d$). From $e'z\bar z-e''t\bar t=\epsilon$,
        we see that $\gcd(e',t\bar t)=1=\gcd(e'',z\bar z)$. Hence
        \begin{displaymath}
          \gcd(e',x_1')=\gcd(e',x_1-x_2)\quad\text{and}\quad\gcd(e'',x_1')=\gcd(e'',x_1),
        \end{displaymath}
        hence~\eqref{eq:dashxidash}. When $\xi$ has order~3, then $x_2=2x_1+1$ by Lemma~\ref{lem:xiorder3conds}.
      \end{proof}
      \begin{proposition}Let $\xi=x_1+x_2\myj+x_3\phi+x_4\myj\phi\in\cO$ have order~3. Assume that
        $x_3^2-x_3x_4+x_4^2$ is divisible by~3. Then there is an Atkin-Lehner element $w_{e'}$ such that
        $w_{e'}\xi w_{e'}^{-1}=\myj$.
      \end{proposition}
      \begin{proof}Let $d_\xi'=\gcd(d,x_1)$ and $d_\xi''=\gcd(d,x_1+1)$. Then $d_\xi'd_\xi''=d$ by Proposition~\ref{prop:xistarconjugatetoxi}.
        Let $e'=d_\xi''$ and $e''=d_\xi'$. Then $e'$ is a Hall divisor of~$d$. By Theorem~\ref{thm:bezoutresult}
        there exist $z,t\in\cO$ such that $e'z\bar z-e''t\bar t\in\{-1,1\}$. By Lemma~\ref{lem:AtkinLehnerconstruction},
        $w_{e'}=e'z+t\phi$ is an Atkin-Lehner element. Let $\xi'=w_{e'}\xi w_{e'}^{-1}$. By \eqref{eq:dashxidash},
        $d_{\xi'}'=\gcd(d,x_1+1)\gcd(d,x_1)=d_\xi''d_\xi'=d$. So $(d_{\xi'}',d_{\xi'}'')=(d,1)=(d_\myj',d_\myj'')$, and
        so by Proposition~\ref{prop:mainconjugacyresult}, there is an~$\alpha\in\cO^\times$ such that $\xi'=\alpha j\alpha^{-1}$.
        So $w\xi w^{-1}=j$ for $w=\alpha^{-1}w_{e'}$, which is an Atkin-Lehner element by the last sentence in the
        statement of Lemma~\ref{lem:AtkinLehnerconstruction}.
        \end{proof}
      \end{subsection}
\end{section}
\begin{section}{Proofs of Theorems~\ref{thm:bezoutresult} and~\ref{thm:Cdcount}.}
  
  Suppose that $d\ge1$ is an integer which has
  a factorization $d=p_1^{m_1}\cdots p_r^{m_r}$, where $p_1,\ldots,p_r$ are distinct primes and $m_1,\ldots,m_r\ge1$.
  We proved Theorem~\ref{thm:bezoutresult}(a) in Section~\ref{sec:ALelts}. Let $d'$ be a Hall divisor of $d$, and let $d''=d/d'$.
  \begin{proof}[Proof of Theorem~\ref{thm:bezoutresult}(b)]Suppose that $d'd''\equiv 2$ (mod~3). Let us prove that one can choose $u,v\in\cL$ such that $d'u-d''v\in \{-1,1\}$ and $3\nmid uv$. Interchanging $d'$ and~$d''$ if necessary, we may suppose that $d'\equiv1$ and $d''\equiv2$ (mod~3).
    Let $d=d'd''$, and form the quaternion algebra~$\cH$ as above.
    By Theorem~\ref{thm:bezoutresult}(a), there exist~$u,v\in\cL$ so that $d'u-d''v=\epsilon\in\{-1,1\}$.
    Modulo~3, $(u,v)\equiv(0,0)$, $(0,1)$, $(1,0)$ or~$(1,1)$. The case $(u,v)\equiv(0,0)$ is clearly impossible,
    and if we can find $u,v\in\cL$ such that $(u,v)\equiv(1,1)$, then Theorem~\ref{thm:bezoutresult}(b) holds
    for~$(d',d'')$. Let us therefore suppose that $(u,v)\equiv(1,0)$ or $(u,v)\equiv(0,1)$ for all such $(u,v)$. If $(u,v)\equiv(1,0)$, then $1\cdot1-2\cdot0\equiv\epsilon$ and so $\epsilon=1$ must
    hold. If $(u,v)\equiv(0,1)$, then $1\cdot0-2\cdot1\equiv\epsilon$, and so $\epsilon=1$ must again hold.

    Pick some $u=x\bar x$ and $v=y\bar y$ such that $d'u-d''v=1$. If $\alpha=a+b\phi\in\cO^\times$, then
    $\alpha(d'x+y\phi)=d'x'+y'\phi$ for $x'=ax+d''b\bar y$ and $y'=ay+d'b\bar x$. One calculates that
    $d'x'\bar x'-d''y'\bar y'=\myNrd(\alpha)$. But by the previous paragraph, $d'x'\bar x'-d''y'\bar y'=1$
    must hold. So $\myNrd(\alpha)=1$ must hold for all $\alpha\in\cO^\times$. 
    But by  \cite[Lemma~39]{Roulleau}, since $d$ is coprime to $3$, the order $\cO$  has elements of norm $-1$, a contradiction.
\end{proof}
  \begin{proof}[Proof of Theorem~\ref{thm:bezoutresult}(c)]Suppose that $d',d''\ge1$ are coprime, and that $3\nmid d'd''$.
    Let us prove that one can choose $u,v\in\cL$ such that $d'u-d''v=1$.

    When $d'\equiv1$ (mod~3), then since $\gcd(d',3d'')=1$,
    Theorem~\ref{thm:bezoutresult}(a) shows that there exist $u,v\in\cL$ so that $d'u-(3d'')v=\epsilon$
    for some $\epsilon\in\{-1,1\}$. Reducing modulo~3, we have $u\equiv\epsilon$ (mod~3), and
    so $\epsilon=1$ must hold, and $d'u-d''(3v)=1$.

    When $d'\equiv2$  and $d''\equiv2$ (mod~3), then again Theorem~\ref{thm:bezoutresult}(a) shows that
    there exist $u,v\in\cL$ so that $(3d')u-d''v=\epsilon$
    for some $\epsilon\in\{-1,1\}$, and reducing modulo~3, we have $v\equiv\epsilon$ (mod~3), and
    so $\epsilon=1$ must hold, and $d'(3u)-d''v=1$.

    When $d'\equiv2$  and $d''\equiv1$ (mod~3), then 
    Theorem~\ref{thm:bezoutresult}(b) shows there exist $u,v\in\cL$ so that $d'u-d''v=\epsilon$
    and so that $3\nmid uv$. Reducing modulo~3, we have $1\equiv d'-d''\equiv d'u-d''v\equiv\epsilon$, and
    so $\epsilon=1$ must hold, and $d'u-d''v=1$.
  \end{proof}

    In view of Proposition~\ref{prop:xistarconjugatetoxi}, parts (a) and~(c)(i) of Theorem~\ref{thm:Cdcount}
  follow from the next result.
    \begin{proposition}\label{prop:ddashddoubledasheqd}There are exactly $2^r$ $\cO^\times$-conjugacy classes of elements
      $\xi\in\cO$ of order~3 for which $d_\xi'd_\xi''=d$.
    \end{proposition}
    \begin{proof}Suppose that $d=d'd''$, where $\gcd(d',d'')=1$. There are exactly $2^r$ such factorizations of~$d$.
      By Theorem~\ref{thm:bezoutresult}(a) there exist $u,v\in\cL$ such that $d'u-d''v=\epsilon\in\{-1,1\}$.

      Suppose that $\epsilon=-1$. Let $x_1=d'u$. Then $x_1+1=d'u+1=d''v$. Then
      $\gcd(d,x_1)=\gcd(d'd'',d'u)=d'\gcd(d'',u)=d'$. Similarly, $\gcd(d,x_1+1)=d''$. Also, 
      $3x_1(x_1+1)=3(d'u)(d''v)=d(3uv)$, and $3uv\in\cL$ because $u$, $v$ and 3 are in~$\cL$. So we can write
      $3uv=x_3^2-x_3x_4+x_4^2$ for some integers~$x_3$ and~$x_4$. Then $\xi=x_1+(2x_1+1)j+x_3\phi+x_4\myj\phi\in\cO$
      has order~3, by Lemma~\ref{lem:xiorder3conds}, and $(d_\xi',d_\xi'')=(d',d'')$.

      Suppose instead that $\epsilon=+1$. Let $y_1=d''v$. Then $y_1+1=d'u$. As in the previous paragraph,
      we get an element $\eta$ of~$\cO$ so that $(d_\eta',d_\eta'')=(d'',d')$. Now let $\xi=\eta^2$. Then $(d_\xi',d_\xi'')=(d',d'')$.

      So for each of the $2^r$ pairs $(d',d'')$ for which $d=d'd''$ and $\gcd(d',d'')=1$,
      there is at least one $\xi\in\cO$ of order~3 with $(d_\xi',d_\xi'')=(d',d'')$. By Propositions~\ref{prop:xistarconjugatetoxi}
      and~\ref{prop:mainconjugacyresult}, any two such $\xi$'s are $\cO^\times$-conjugate.
    \end{proof}

    We now consider the $\xi\in\cO$ of order~3 for which $d_\xi'd_\xi''=d/3$.

    \begin{lemma}\label{lem:ddashddoubledasheqtilded}Suppose that $d$ is divisible by~3, and write $\tilde d=d/3$. Then
      the following are equivalent:
    \begin{itemize}
    \item[(a)]There is a $\xi\in\cO$ of order~3 so that $d_\xi'd_\xi''=\tilde d$,
    \item[(b)]There is a $\xi\in\cO$ of order~3 so that $(d_\xi',d_\xi'')=(\tilde d,1)$,
    \item[(c)]For any coprime $d',d''\ge1$ such that $d'd''=\tilde d$, there is a $\xi\in\cO$
      of order~3 so that $(d_\xi',d_\xi'')=(d',d'')$.
      \end{itemize}
  \end{lemma}
  \begin{proof} (a)$\implies$(b): Suppose that $\xi=x_1+x_2\myj+x_3\phi+x_4\myj\phi\in\cO$ has order~3
    and $d_\xi'd_\xi''=\tilde d$. In view of Remark~\ref{rem:effectofsquaring},
    replacing $\xi$ by $\xi^2$ if necessary, we may suppose that 3 does not
    divide~$x_1+1$, and hence 3 does not divide~$d_\xi''$. Let $e'=d_\xi''$ and $e''=d/e'=3d_\xi'$. Then
    $\gcd(e',e'')=1$ and $e'e''=d$. By Lemma~\ref{lem:AtkinLehnerconstruction}, there is an Atkin-Lehner element $w=w_{e'}$
    associated with~$e'$. Let $\eta=w\xi w^{-1}$. Then by Lemma~\ref{lem:atkinlehneraction},
    $d_\eta'=\gcd(e',x_1+1)\gcd(e'',x_1)$. Now $e'=d_\xi''$ is a divisor
    of~$x_1+1$, and so $\gcd(e',x_1+1)=e'=d_\xi''$. Also, $e''=3d_\xi'$ is a divisor of~$3\tilde d=d$,
    and so $\gcd(e'',x_1)$ is a divisor of~$\gcd(d,x_1)=d_\xi'$. As $d_\xi'$ divides~$x_1$ and $3d_\xi'=e''$,
    we have $\gcd(e'',x_1)=d_\xi'$.
    Hence $d_\eta'=d_\xi'd_\xi''=\tilde d$. Applying Lemma~\ref{lem:atkinlehneraction} to $\eta^2=w\xi^2w^{-1}$, we have
    $d_\eta''=d_{\eta^2}'=\gcd(e',x_1)\gcd(e'',x_1+1)$. Now $e'=d_\xi''$, is a divisor of~$x_1+1$, and so
    $\gcd(e',x_1)=1$, while $\gcd(e'',x_1+1)=\gcd(3d_\xi',x_1+1)=1$ because 3 does not divide $x_1+1$ and
    $d_\xi'$ is a divisor of~$x_1$. Hence $(d_\eta',d_\eta'')=(\tilde d,1)$.

    (b)$\implies$(c): Suppose that $\xi=x_1+x_2\myj+x_3\phi+x_4\myj\phi\in\cO$ has order~3
    and $(d_\xi',d_\xi'')=(\tilde d,1)$. Suppose that $d',d''\ge1$ are coprime, and that $d'd''=\tilde d$.
    In view of Remark~\ref{rem:effectofsquaring}, we may suppose that 3 does not divide $d'$.
    Let $e'=3d''$, and let $e''=d'$. Then $\gcd(e',e'')=1$ and $e'e''=d$. Let $w=w_{e'}$ be an
    Atkin-Lehner element associated with~$e'$, and let $\eta=w\xi w^{-1}$. By Lemma~~\ref{lem:atkinlehneraction}, $d_\eta'=\gcd(e',x_1+1)\gcd(e'',x_1)$.
    Now $\gcd(e'',x_1)=\gcd(d',x_1)=d'$ because $d'd''=\tilde d=d_\xi'$ divides~$x_1$. Also,
    $\gcd(e',x_1+1)=\gcd(3d'',x_1+1)$ divides $\gcd(3\tilde d,x_1+1)=\gcd(d,x_1+1)=d_\xi''=1$. So $d_\eta'=d'$.
    Now applying Lemma~\ref{lem:atkinlehneraction} to $\eta^2=w\xi^2w^{-1}$, we have
    $d_\eta''=d_{\eta^2}'=\gcd(e',x_1)\gcd(e'',x_1+1)$. Now $e'=3d''$, and $d''$ divides $\tilde d=\gcd(d,x_1)$.
    So $d''$ divides $\gcd(3d'',x_1)$. But $3d''$ does not divide~$x_1$, as otherwise (as $\gcd(d',3d'')=1$),
    $d$ would divide~$x_1$, contradicting $\gcd(d,x_1)=\tilde d$. So $\gcd(e',x_1)=d''$. Finally,
    $\gcd(e'',x_1+1)=\gcd(d',x_1+1)$ is a divisor of $\gcd(d,x_1+1)=d_\xi''=1$. So $d_\eta''=d''$, and
    so $(d_\eta',d_\eta'')=(d',d'')$.    
   
    (c)$\implies$(a): This is obvious.
  \end{proof}

  \begin{proof}[Proof of Theorem~\ref{thm:Cdcount}(c)(ii)]
    Assume $3\mid d$, write $\tilde d=d/3$, and suppose that $\tilde d\equiv2$ (mod~3).
    Let us prove that $C_d=2^{r+1}$, where $r$ is the number of prime divisors of~$d$.
      By Theorem~\ref{thm:bezoutresult}(b), there exist $u,v\in\cL$ so that $\tilde du-v=\epsilon\in\{-1,1\}$ and so that
      $3\nmid uv$. Then $\epsilon\equiv2\cdot1-1=1$ (mod~3), so $\epsilon=1$. Let $x_1=v$. Then 
      $3x_1(x_1+1)=3\tilde duv=duv$. We can write $uv=x_3^2-x_3x_4+x_4^2$. Then $\xi=x_1+x_2\myj+x_3\phi+x_4\myj\phi\in\cO$
      has order~3, and $d_\xi'=\gcd(x_1,d)=\gcd(v,3\tilde d)=1$, while $d_\xi''=\gcd(x_1+1,d)=\gcd(\tilde du,3\tilde d)=\tilde d$.
      So $(d_\xi',d_\xi'')=(1,\tilde d)$. By the (a)$\implies$(c) part of Lemma~\ref{lem:ddashddoubledasheqtilded}, for each
      of the $2^{r-1}$ pairs $(d',d'')$ of positive integers such that $\gcd(d',d'')=1$ and $d'd''=\tilde d$, there is an $\eta\in\cO$
      of order~3 such that $(d_\eta',d_\eta'')=(d',d'')$. For each such~$\eta$, $\eta^*=\phi\xi^2\phi^{-1}$ has the same property,
      and is not conjugate to~$\eta$, by Proposition~\ref{prop:xistarconjugatetoxi}. So we have exactly $2\times 2^{r-1}$ conjugacy classes of elements~$\eta$ of order~3 such
      that $d_\eta'd_\eta''=\tilde d$. Add to these the $2^r$ conjugacy classes provided by Proposition~\ref{prop:ddashddoubledasheqd},
      and we have in total $2\times2^r$  conjugacy classes.
  \end{proof}

  \begin{proof}[Proof of Theorem~\ref{thm:Cdcount}(b)]
     Suppose that $9\mid d$.  Let us prove that $C_d=2^r$ or~$3\times2^r$, where $r$ is the number of prime divisors of~$d$.
    By Proposition~\ref{prop:ddashddoubledasheqd}, there are exactly~$2^r$ conjugacy
      classes of $\xi$'s for which $d_\xi'd_\xi''=d$. So if $C_d\ne2^r$, there must be a $\xi$ of order~3
      so that $d_\xi'd_\xi''=\tilde d$.  Now $\tilde d=d/3$ also has $r$ distinct prime factors. By the (a)$\implies$(c) part of Lemma~\ref{lem:ddashddoubledasheqtilded}, for each
        of the $2^r$ pairs $(d',d'')$ of coprime positive integers such that $d'd''=\tilde d$,
        there is an $\xi\in\cO$ of order~3 such that $(d_\xi',d_\xi'')=(d',d'')$. For each such~$\xi$, $\xi^*=\phi\xi^2\phi^{-1}$ has the same property,
        and is not conjugate to~$\xi$, by Proposition~\ref{prop:xistarconjugatetoxi}. So we have exactly $2\times 2^r$ conjugacy classes of
        elements~$\xi$ of order~3 such that $d_\xi'd_\xi''=\tilde d$. Add to these the $2^r$ conjugacy classes
        provided by Proposition~\ref{prop:ddashddoubledasheqd}, and we have in total $3\times2^r$  conjugacy classes.
  \end{proof}

\end{section}
\begin{section}{The conjectures}\label{sec:conjectures}

  \begin{proposition}\label{prop:compareconjectures} Suppose that $d$ is divisible by~9, and write $\tilde d=d/3$.
    Then the following are equivalent:
    \begin{itemize}
    \item[(a)] Conjecture~2 is true for~$d$,
    \item[(b)] Conjecture~1 is true for all coprime $d',d''\ge1$ such that $d'd''=\tilde d$,
    \item[(c)] Conjecture~1 is true for some coprime $d',d''\ge1$ such that $d'd''=\tilde d$,
    \item[(d)] There is a $\xi\in\cO(d)$ of order~3 so that $d_\xi'd_\xi''=\tilde d$.
    \end{itemize}
  \end{proposition}
  \begin{proof}Note that $\tilde d$, like~$d$, has $r$ distinct prime divisors, one of which is~3.
      So there are exactly~$2^r$ pairs $(d',d'')$ such that $d'd''=\tilde d$ and $\gcd(d',d'')=1$.     

    (c)$\implies$(d): Suppose that $d'd''=\tilde d$ and $\gcd(d',d'')=1$, and that
    Conjecture~1 is true for $(d',d'')$. Then there exist $u,v\in\cL$ such
    that $d'u-d''v=\epsilon\in\{-1,1\}$ and $3\nmid uv$. If $\epsilon=1$, let $x_1=d''v$. Then
  \begin{displaymath}
        \frac{3x_1(x_1+1)}{d}=\frac{3(d''v)(d'u)}{3d'd''}=uv,
      \end{displaymath}
  which is in~$\cL$ and not divisible by~3. We can write $uv=x_3^2-x_3x_4+x_4^2$, and then
  $\xi=x_1+(2x_1+1)\myj+x_3\phi+x_4\myj\phi$
      has order~3. Also, $d_\xi'=\gcd(d,x_1)=\gcd(3d'd'',d''v)=d''$ because $\gcd(3d',v)=1$, and
      $d_\xi''=\gcd(d,x_1+1)=\gcd(3d'd'',d'u)=d'$, because $\gcd(3d'',u)=1$. So $d_\xi'd_\xi''=d''d'=\tilde d$.
      Similarly, if $\epsilon=-1$, we take $x_1=d'u$. 

      (d)$\implies$(a): Suppose there is a $\xi\in\cO$ of order~3 so that
      $d_\xi'd_\xi''=\tilde d$. By Proposition~\ref{prop:ddashddoubledasheqd}, there are already $2^r$
      conjugacy classes of elements $\eta\in\cO$ of
      order~3 such that $d_\eta'd_\eta''=d$, and so $C_d>2^r$. Hence $C_d=3\times2^r$
      by Theorem~\ref{thm:Cdcount}(b). So Conjecture~2 is true for~$d$.

      (a)$\implies$(b): Suppose that Conjecture~2 holds for~$d$. By Proposition~\ref{prop:ddashddoubledasheqd},
      there are only $2^r$ $\cO^\times$-conjugacy classes of elements~$\xi$ of order~3 for which $(d_\xi',d_\xi'')=d$.
      The remaining conjugacy classes of elements of order~3
      must consist of elements~$\xi$ for which $(d_\xi',d_\xi'')=\tilde d$. As Conjecture~2 is true for~$d$, there are $2\times2^r$
      of these remaining conjugacy classes. But there are only $2^r$ pairs $(d',d'')$ of coprime
      integers whose product is~$\tilde d$, and by Proposition~\ref{prop:mainconjugacyresult} for each such pair there are at most two
      conjugacy classes of elements $\xi$ for which $(d_\xi',d_\xi'')=(d',d'')$. So Conjecture~2 being true
      for~$d$ implies that for each such pair $(d',d'')$
      there are indeed two conjugacy classes of elements~$\xi$ for which $(d_\xi',d_\xi'')=(d',d'')$.

       Now suppose that $d'd''=\tilde d$ and $\gcd(d',d'')=1$. To show that Conjecture~1 holds for $(d',d'')$, we may
      assume that $3\mid d'$. 
      Pick $\xi=x_1+(2x_1+1)\myj+x_3\phi+x_4\myj\phi$ of order~3 so that $(d_\xi',d_\xi'')=(d',d'')$.
      By Proposition~\ref{prop:xistarconjugatetoxi}, $d_\xi'd_\xi''=\tilde d$ implies that 3 does not divide~$x_3^2-x_3x_4+x_4^2$,
      and $x_1(x_1+1)=\tilde d(x_3^2-x_3x_4+x_4^2)$ by~\eqref{eq:xiorder3cond}.

      Write $d=3^kd^*$, where $k\ge2$ and $3\nmid d^*$. Then $3^{k-1}\mid x_1(x_1+1)$. But $3\mid d'$, $d'=d_\xi'$ and $d_\xi'\mid x_1$.
      So $3\mid x_1$, and we can write $x_1=3^{k-1}x_1'$, where $3\nmid x_1'$. Then 
      \begin{displaymath}
        d'=d_\xi'=\gcd(d,x_1)=\gcd(3^kd^*,3^{k-1}x_1')=3^{k-1}\gcd(3d^*,x_1')=3^{k-1}\gcd(d^*,x_1'),
        \end{displaymath}
      and $d''=d_\xi''=\gcd(d,x_1+1)=\gcd(3^kd^*,x_1+1)=\gcd(d^*,x_1+1)$. As $d^*$ divides~$x_1(x_1+1)$ and $\gcd(x_1,x_1+1)=1$, we can write $d^*=d_1d_2$,
      where $d_1=\gcd(d^*,x_1)$ and $d_2=\gcd(d^*,x_1+1)$. So $d'=3^{k-1}d_1$ and $d''=d_2$.
         There are integers~$u,v$ for which $x_1'=d_1u$ and $x_1+1=d_2v$.      So
      \begin{displaymath}
        uv=\frac{3(3^{k-1}d_1u)(d_2v)}{3^kd_1d_2}=\frac{3x_1(x_1+1)}{d}=x_3^2-x_3x_4+x_4^2\in\cL,
        \end{displaymath}
      and $3\nmid uv$. Now $d'u=x_1$ and $d''v=x_1+1$ implies that $\gcd(u,v)=1$. If $u,v\ge0$, this and $uv\in\cL$ implies that $u,v\in\cL$,
      and $d'u-d''v=x_1-(x_1+1)=-1$. If $u,v\le0$, then $-u,-v\in\cL$ and $d'(-u)-d''(-v)=1$. So Conjecture~1 is true for~$(d',d'')$,
      and (b)~holds.

      (b)$\implies$(c): This is obvious.      
  \end{proof}
   Let us write $\cH=\cH(d)=\{x+y\phi_d:x,y\in\Q(\myj)\}$
  and~$\cO(d)=\{x+y\phi_d:x,y\in\Z[\myj]\}$ (where $\phi_d=\phi$ satisfies $\phi_d ^2=d$). Fix a $\theta\in\Z[\myj]$ such that $\theta\bar\theta=3$.
  The map $f:x+y\phi_{3d}\mapsto x+\theta y\phi_d$ is an algebra isomorphism $\cH(3d)\to\cH(d)$, with
  $f(\cO(3d))\subsetneqq\cO(d)$ and $f(\cO^\times(3d))\subset\cO^\times(d)$.
  We choose $\theta=1-\myj$ below.

  Assume that $d\ge1$ has exactly $r$ distinct prime factors. Let's denote by $\cG$ the
  set of~$d\ge1$ for which either $3^2\nmid d$ (so that $C_d$ is as stated in
  Theorem~\ref{thm:Cdcount}(a) and~(c)), or $3^2\mid d$ and $C_d=3\times2^r$, so that
  Conjecture~2 is true for~$d$. In particular,
  \begin{itemize}
  \item[(i)] $3\nmid d\implies d\in\cG$, and $C_d=2^r$;
  \item[(ii)] $3^1\|d\implies d\in\cG$, and $C_d=2^r$ or~$C_d=2^{r+1}$ according as $\tilde d\equiv1$ or $\tilde d\equiv2$ (mod~3).
  \end{itemize}
  \begin{lemma}\label{lem:isomorphismequivs}The following are equivalent:
    \begin{itemize}
    \item[(a)] $f$ maps $\cO^\times(3d)$ onto~$\cO^\times(d)$,
    \item[(b)] $C_{3d}=C_d=2^r$,
    \item[(c)] $3\mid d$, and there are no $u,v\in\cL$ such that
      $3\nmid uv$ and $u-dv\in\{-1,1\}$.
    \item[(d)] $3\mid d$, and $3d\not\in\cG$.
    \end{itemize}
  \end{lemma}
  
  \begin{proof}(a)$\implies$(b): When $f$ maps $\cO^\times(3d)$ onto~$\cO^\times(d)$, it gives
    an isomorphism $\cO^\times(3d)\to\cO^\times(d)$. So $C_{3d}=C_d$ must hold. If $3\nmid d$
    then $d\in\cG$, and $C_d=2^r$, and as $3d$ then has $r+1$ distinct prime
    factors, $C_{3d}\ge2^{r+1}$ by Proposition~\ref{prop:ddashddoubledasheqd}. So $d$ must be divisible by~3.

    If $\xi=x_1+(2x_1+1)\myj+x_3\phi_{3d}+x_4\myj\phi_{3d}\in\cO(3d)$ has order~3, then
    $f(\xi)=x_1'+(2x_1'+1)\myj+x_3'\phi_{d}+x_4'\myj\phi_{d}$ for $x_1'=x_1$, $x_3'=x_3+x_4$ and $x_4'=2x_4-x_3$.
    Writing $\xi^*=x_1+(2x_1+1)\myj+(x_4-x_3)\phi_{3d}+x_4\myj\phi_{3d}$ as usual, notice that
    $f(\xi^*)=x_1'+(2x_1'+1)\myj+x_4'\phi_{d}+x_3'\myj\phi_{d}$, which is conjugate to~$f(\xi)$ by
    Lemma~\ref{lem:interchangex3x4}. If $C_{3d}>2^r$, then by Proposition~\ref{prop:ddashddoubledasheqd}
    there exist $\xi\in\cO^\times(3d)$ such that $\xi$ is not conjugate to~$\xi^*$, and so $f$ cannot
    be an isomorphism $\cO^\times(3d)\to\cO^\times(d)$. So $C_{3d}=2^r$ must hold.

    (b)$\implies$(c): If (b) holds, then as in the first paragraph, 3 must divide~$d$.
    Suppose that there exist $u,v\in\cL$ so that $u-dv=\epsilon\in\{-1,1\}$ and $3\nmid uv$. Then $\epsilon=1$ must hold.
    Let $x_1=dv$ and then $3x_1(x_1+1)=3duv=(3d)(x_3^2-x_3x_4+x_4^2)$ for some integers $x_3,x_4$.
    Then $\xi=x_1+(2x_1+1)\myj+x_3\phi_{3d}+x_4\myj\phi_{3d}\in\cO(3d)$ has order~3, and the fact that 3 does not
    divide $uv=x_3^2-x_3x_4+x_4^2$ implies that $\xi^*$ is not conjugate to~$\xi$, by Proposition~\ref{prop:xistarconjugatetoxi}.
    So $C_{3d}>2^r$, contradicting~(b). 

    (c)$\implies$(a): Suppose that (c)~holds. By Theorem~\ref{thm:bezoutresult}(a), there exist $u,v\in\cL$
    such that $u-dv=\epsilon\in\{-1,1\}$. Since $3\mid d$, $\epsilon=1$ and $3\nmid u$ must hold.
    Write $u=x\bar x$ and $v=y\bar y$. By (c), $v$ must be divisible by~3.
    Thus $y$ must be in~$(1-\myj)\mathfrak o$, and $x\not\in(1-\myj)\mathfrak o$. Now let $\alpha=a+b\phi\in\cO^\times(d)$.
    Then $\alpha(x+y\phi)=x'+y'\phi$ for $x'=ax+b\bar yd$ and $y'=ay+b\bar x$. One calculates that
    $x'{\bar x}'-dy'{\bar y}'=\epsilon\myNrd(\alpha)=\myNrd(\alpha)\in\{-1,1\}$ and so
    $y'\in(1-\myj)\mathfrak o$ must hold. This implies that $b\in(1-\myj)\mathfrak o$ for all $a+b\phi\in\cO^\times(d)$,
    and so each $\alpha\in\cO^\times(d)$ is in $f(\cO^\times(3d))$. So (a)~holds.

    (c)$\iff$(d): There do not exist $u,v\in\cL$
    such that $3\nmid uv$ and $u-dv=\epsilon\in\{-1,1\}$, if and only if $3d\not\in\cG$, by
    the (a)$\iff$(b) part of Proposition~\ref{prop:compareconjectures}.
  \end{proof}

  \begin{corollary}\label{cor:multby3}Suppose that $d$ is divisible by~9, and that $d\in\cG$. Then $3d\in\cG$.
  \end{corollary}
  \begin{proof}By hypothesis, $C_d=3\times2^r$. By (b)$\iff$(c) in Lemma~\ref{lem:isomorphismequivs},
    there exist $u,v\in\cL$ so that $u-dv\in\{-1,1\}$ so that $3\nmid uv$. So by
    (c)$\implies$(a) in Proposition~\ref{prop:compareconjectures}, $3d\in\cG$.
  \end{proof}

  \begin{corollary}\label{cor:dstar2mod3}Suppose that $d=3^kd^*$, where $d^*\equiv2$ (mod~3). Then $d\in\cG$.
  \end{corollary}
  \begin{proof}When $k=0$ or~1, $d\in\cG$ by~(i) and~(ii) above. We use induction. So suppose that $k\ge2$ and
    that $d=3^{k-1}d^*\in\cG$. Then either $C_d=3\times2^r$ (if $k\ge3$)
    or $C_d=2\times2^r$ (if $k=2$). Either way, $C_d>2^r$, and so condition~(b) in Lemma~\ref{lem:isomorphismequivs}
    does not hold. So by (b)$\iff$(d) in Lemma~\ref{lem:isomorphismequivs}, $3d=3^kd^*\in\cG$.
  \end{proof}

  \begin{corollary}\label{cor:dstarloeschian}Suppose that $d\in\cL$. Then $d\in\cG$.
  \end{corollary}
  \begin{proof}Write $d=3^k v_0$, where $k\ge0$ and $3\nmid v_0$.
    If $k=0$, then $d\in\cG$ by~(i) above.
    If $k=1$, then $d\in\cG$ by~(ii) above.

    Now suppose that $k\ge2$. If $d\not\in\cG$
    then $d_1=dv_1\not\in\cG$ for any Loeschian~$v_1$ such that $3\nmid v_1$. For $\tilde{d_1}=\tilde dv_1=3^{k-1}v_0v_1$, so if
    $u-\tilde dv_1v\in\{-1,1\}$ where $u,v\in\cL$ and $3\nmid uv$, then $u-\tilde d(v_1v)\in\{-1,1\}$ and
    $3\nmid v_1v$, so that by Proposition~\ref{prop:compareconjectures}, $d\in\cG$,
    contrary to hypothesis. In particular, if $d\not\in\cG$, then $dv_0=3^k(v_0)^2\not\in\cG$, and so $9v_0^2\not\in\cG$,
    by Corollary~\ref{cor:multby3}. But $d=9c^2\in\cG$ for any integer~$c\ge1$.
    For Lemma~\ref{lem:xiorder3conds} shows that $\xi=3c^2+(6c^2+1)\myj+(c-1)\phi+2c\myj\phi\in\cO$ has order~3,
    and $\gcd(9c^2,3c^2)=3c^2$ and $\gcd(9c^2,3c^2+1)=1$, so that $(d_\xi',d_\xi'')=(d/3,1)$.
  \end{proof}

  \begin{corollary}\label{cor:multby3tothenu}For any $d\ge1$, there is a $\nu\ge0$ so that $3^\nu d\in\cG$.
  \end{corollary}
  \begin{proof}Suppose that $3^k\|d$. If $k=0$ or~$k=1$, then we can take $\nu=0$, by~(i)
    and~(ii) above. So suppose that $k\ge2$. Then $3^\nu d$ has exactly $r$ distinct
    prime divisors, and $3^\nu d\in\cG$ if and only if $C_{3^\nu d}>2^r$.
    
    Suppose $\xi=x_1+x_2\myj+x_3\phi+x_4\myj\phi\in\cO(d)$ has order~3, that $\xi\ne\myj,\myj^2$, and
    that $3^\nu\|(x_3^2-x_3x_4+x_4^2)$. Then there are integers $\hat x_3$ and $\hat x_4$ so that
  \begin{displaymath}
    3x_1(x_1+1)=(3^\nu d)({\hat x_3}^2-{\hat x_3}{\hat x_4}+{\hat x_4}^2).
  \end{displaymath}
  So $\eta=x_1+(2x_1+1)\myj+\hat x_3\phi_{3^\nu d}+\hat x_4\myj\phi_{3^\nu d}\in\cO(3^\nu d)$ has order~3,
  and ${\hat x_3}^2-{\hat x_3}{\hat x_4}+{\hat x_4}^2$ is not divisible by~3. Hence
  $C_{3^\nu d}>2^r$, and so $3^\nu d\in\cG$.
  \end{proof}

  \begin{remark}To show that Conjecture~2 is true for all~$d$, it is enough to prove it for $d=9d^*$,
      where $d^*\equiv1$ (mod~3). We have verified this for $d^*\le10^9$, by finding $u,v\in\cL$ with
      $3\nmid uv$ and $u-3d^*v=1$.
       By Corollary~\ref{cor:dstar2mod3}, the conjecture is true when $d_\cH\equiv2$ (mod~3) (in the notation of the Introduction), because
     in general $d=3^kd_\cH v$, where $k\ge0$ and $v\equiv1$ (mod~3).
    Other integers~$d\ge1$ for which the conjecture is true are those for which $9\mid d$, $\tilde d=d/3$ is not
    a perfect square, and the smallest non-trivial solution $(x_0,y_0)$ of Pell's equation $X^2-\tilde d\,Y^2=1$
    satisfies $3\nmid y_0$.
    \end{remark}

\end{section}
\bibliographystyle{amsplain}

\vspace{2mm}
 
{\small  
\noindent Donald I. Cartwright
\\School of Mathematics and Statistics,
\\University of Sydney,
\\New South Wales 2006, 
\\Australia 
\vspace{1mm}
\\ \email{donald.cartwright@gmail.com}
\vspace{3mm}
\\\noindent Xavier Roulleau,
\\Aix-Marseille Universit\'e, CNRS, I2M UMR 7373,  
\\13453 Marseille,
\\ France
\vspace{1mm}
\\ \email{Xavier.Roulleau@univ-amu.fr}
\begin{verbatim}
http://www.i2m.univ-amu.fr/perso/xavier.roulleau/Site_Pro/Bienvenue.html
\end{verbatim}
}
\end{document}